\documentclass{article}
\usepackage{amsfonts,dsfont,amsmath,amsthm,amssymb,enumerate,color,cite,bbm}
\usepackage{graphicx}
\usepackage{fullpage}

\theoremstyle{plain}
\newtheorem{theorem}{Theorem}
\newtheorem{lemma}{Lemma}
\newtheorem{proposition}{Proposition}

\theoremstyle{definition}
\newtheorem{definition}{Definition}

\theoremstyle{remark}
\newtheorem{remark}{Remark}

\numberwithin{equation}{section}

\newcommand{\bfn}{{\bf{n}}}
\newcommand{\bfe}{{\bf{e}}}
\newcommand{\bfu}{{\bf{u}}}
\newcommand{\bfw}{{\bf{w}}}
\newcommand{\bfv}{\bf{v}}

\newcommand{\bfnu}{\boldsymbol{\nu}}
\newcommand{\bftheta}{{\boldsymbol{\theta}}}
\newcommand{\bfTheta}{{\boldsymbol{\Theta}}}
\newcommand{{\bfzero}}{\boldsymbol{0}}
\newcommand{\bfnabla}{\boldsymbol{\nabla}}
\newcommand{\bfxi}{\boldsymbol{\xi}}
\newcommand{\bfPsi}{\boldsymbol{\Psi}}
\newcommand{\bfPhi}{\boldsymbol{\Phi}}
\newcommand{\bfchi}{\boldsymbol{\chi}}
\newcommand{\bfphi}{\boldsymbol{\phi}}

\newcommand{\qp}{\overset{2-{\bftheta}}{\rightharpoonup}}

\newcommand{\xoe}{\tfrac{x}{\ep}}

\newcommand{\RR}{\mathbb{R}} 
\newcommand{\NN}{\mathbb{N}}
\newcommand{\CC}{\mathbb{C}}
\newcommand{\ZZ}{\mathbb{Z}}

\renewcommand{\exp}{\mathrm{exp}}

\newcommand{\V}{\mathcal{V}}
\newcommand{\I}{\mathcal{I}}
\newcommand{\T}{\mathcal{T}}

\renewcommand{\i}{\rm i}
\newcommand\ep{\varepsilon}

\title{Quasi-periodic two-scale homogenisation and effective spatial dispersion in high-contrast media}
\author{Shane Cooper}
\begin{document}
\maketitle

\begin{abstract}
The convergence of spectra via two-scale convergence for double-porosity models is well known. A crucial assumption in these works is that the stiff component of the body forms a connected set. We show that under a relaxation of this assumption the (periodic) two-scale limit of the operator is insufficient to capture the full asymptotic spectral properties of high-contrast periodic media. Asymptotically, waves of all periods (or quasi-momenta) are shown to persist and an appropriate extension of the notion of  two-scale convergence is introduced. As a result, homogenised limit equations with none trivially quasi-momentum dependence are found as resolvent limits of the original operator family. This results in asymptotic spectral behaviour with a rich dependence on quasimomenta.
\end{abstract}

\section{Introduction}
\hspace{10pt} The model problem to study  time-harmonic waves, with frequency $\omega$,  in  media with microstructure is
\begin{equation*}
-{\rm div} \big( a_\ep( \xoe) \bfnabla u  \big) = \omega^2  u \qquad \text{in $\Omega$}
\end{equation*}
 where the wave $u$  represents the information being propagated, such as pressure in acoustics, deformation in elasticity or electromagnetic fields in electromagnetism\footnote{In elasticity and electromagnetism the wave equation describes certain polarised waves: e.g. Shear polarised wave in elasticity or Transverse Electric and Transverse Magnetic polarised waves for the Maxwell system.}. 
  The microstructured nature of the media is characterised by periodic coefficients $a_\ep\ $\footnote{The implied non-trivial dependence of $a$ on $\ep$ is deliberate and, as we shall see, important.} 
:
\begin{equation*}
a_\ep(y) = \left\{ \begin{matrix}
a_{1\ep}(y), & y \in Q_1, \\
a_{0\ep}(y), & y \in Q_0,
\end{matrix} \right.
\end{equation*}
where $a_{0\ep}$, $a_{1\ep}$  are (the square-root of) the wave speeds of the individual constitutive material components,
 see figure \ref{fig:compositedomain}.
\begin{figure}
\centering
\includegraphics[width=0.7\linewidth]{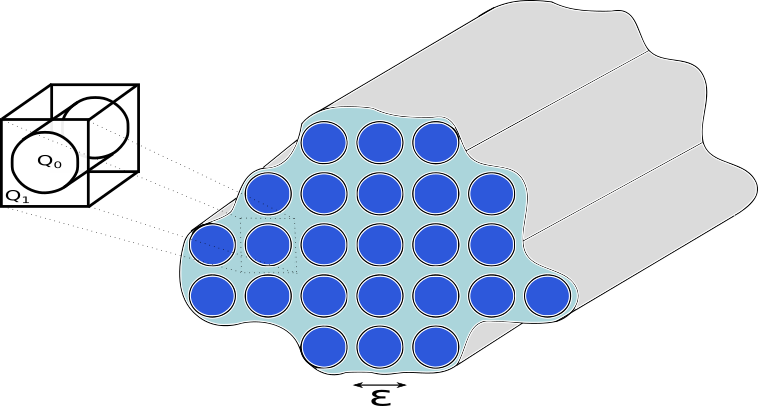}
\caption{A typical schematic of some three-dimensional composite media with period microstructure.}
\label{fig:compositedomain}
\end{figure}
The parameter $\ep$ represents the ratio between the size of the microstructure and the observable length scale, and is  typically taken to be small. From the point of view of applications, it is important to study the asymptotic behaviour of these waves in the limit of vanishing $\ep$.

\hspace{10pt} A classical approximation, provided by the homogenisation theorem\footnote{Also called the long-wavelength or quasi-static approximation depending on the community.}, states that for fixed frequency $\omega$ the 
 microstructured media 
can be approximated by 
an `effective' homogeneous media whose wave speed $a^{\rm hom}$ is constant and determined directly from the `local periodic' behaviour of the problem. The intuition behind why the homogenisation theorem holds is that the `wavelength' of $u$ is long with respect to the microstructure: variations in $u$ appear over much longer distances than the media's period. Mathematically, this is ensured by assuming that  $a_\ep$ are taken to be uniformly bounded and elliptic with respect to $\ep$, for example
\[
a_{1\ep} = a_1, \quad a_{0\ep} = a_0, \qquad \text{for bounded elliptic $a_0,a_1$.}
\]

\hspace{10pt} 
It has been known for some time now that interesting effects appear when the above elliptic conditions are not uniform. This happens for example in so-called high-contrast media.  In the context of waves, high-contrast media of particular interest are the so-called double porosity models 
which admit the `critical' scaling: 
\[
a_{1\ep} = a_1, \quad a_{0\ep} = \ep^2 a_0.
\]
Physically, this  critical scaling corresponds to the wavelength of $u$ remaining `long' within the media $Q_1$ but in media  $Q_0$ the wavelength is at the `resonant' scale, i.e. of the same order as the size of the microstructure. Thus violating the underlying intuition for the long-wavelength approximation. 

\hspace{10pt} The mathematical analysis of high-contrast problems has given rise to rigorous descriptions of various scale-interaction phenomena  such as memory effects and other non-local effects (e.g. \cite{FeKh,Sa,BeGr,Br,CaMi,CaSe,ChSmZh}). 
Within the context of  wave propagation, an important feature of high-contrast problems is that they contain spectral gaps (cf. \cite{HeLi,Zh1,Zh2}):  frequencies at which no wave can propagate through the underlying medium. Such gaps are important from the point of view of wave-guiding applications such as photonic crystal fibres. 

An important initial work in the study of the spectrum of high-contrast elliptic operators  was undertaken by V. V. Zhikov \cite{Zh1,Zh2}.
Therein,
the homogenisation theory for 
double porosity-type problems was developed within the framework of the so-called two-scale convergence of G. Nguetseng-G. Allaire \cite{Ng,All}. Using this theory, Zhikov derived two-scale limit spectral equations that contain a non-trivial coupling between micro- and macro-scales.
Such a coupling leads to an eigenvalue problem with a highly non-linear spectral dependence, described by a function $\beta$.
The convergence of spectra (in the appropriate sense) was proved and, by doing  so, demonstrates that this $\beta$ function provides an explicit description of the asymptotic structure of the spectrum. 
Such an explicit description of the limit spectral behaviour via two-scale homogenisation has made way for mathematical studies of high-contrast media as wave-guides:
in  \cite{KaSm1}   using multi-scale asymptotics and supplemented with analysis based on two-scale convergence in \cite{cherd}. 
	
\hspace{10pt} Moreover, the Zhikov $\beta$ function was later independently discovered by G. Bouchitt\'{e} and D. Felbacq \cite{BoFe} in the specific context of TM-polarised electromagnetic wave propagation in a dilute dielectric two-dimensional photonic crystal fibre; therein the authors made the interesting interpretation of the $\beta$-function playing the role of effective negative magnetism. 
 Later, in the context of elasticity, a matrix analogue of the $\beta$ function is derived and plays the role of frequency-dependent effective density \cite{Av1,Av2,ZhPa}. 
%
Such works demonstrate that the unusual phenomena observed in high-contrast media can be described by non-standard constitutive laws provided via two-scale homogenisation.

\hspace{10pt} The idea that high-contrast media can result in the appearance of non-standard constitutive laws and give rise to composite media with complex wave phenomena near micro-resonances has prompted a recent energetic pursuit of 
such laws in the contexts of elasticity (e.g. \cite{LiChSh}) and electromagnetism (e.g. \cite{KoSh,ChCo2}). Applications can be found in areas such as cloaking (e.g. \cite{MiBrWi,MiWi}).  It was shown in the work of V. P. Smyshlyaev \cite{SmMoM}, building on related ideas in \cite{ChSmZh}, that the two-scale homogenised limit of various anisotropic elastic media  contain not only the temporal non-locality (as described by the Zhikov $\beta$ function) but also exhibit spatial non-locality. The presence of which leads to the phenomena of `directional' localisation: the number of admissible propagating wave modes depends not only on the frequency but on the direction of propagation. Such a feature is important for cloaking applications. These motivating works have
led to recent systematic study containing rigorous asymptotic and spectral analysis of general mathematical constructions containing `high-contrasts' \cite{IVKVPS}. Analysis based on the work \cite{IVKVPS} has led to the demonstration that the two-scale convergence is insufficient to fully study the spectrum of general high-contrast problems, see \cite[Chapter 5]{Cothesis},\cite{ChCo1dhc}. The reason for this inconsistency is due to the presence of quasi-periodic micro-oscillations that
persist at leading-order  in general high-contrast media.

{\hspace{10pt}} In this work we appropriately develop homogenisation theory to study quasi-periodic micro-oscillations. This is achieved by extending the two-scale convergence framework to admit such
  oscillations.  We  explain the (lack of the) role these  micro-oscillations  in the numerous previous works on high-contrast 
 problems. Then, we  apply this  theory  in the  spectral analysis of  a novel class of  high-contrast media. In particular, we shall show that by relaxing the geometric assumptions in the double-porosity model
leads to multi-scale homogenised models that contain a new feature: the effective wave speed depends on the quasi-momenta in a highly discontinuous fashion. Specifically,  the non-standard constitutive equations for such high-contrast media exhibit spatial dispersion. The presence of this novel feature is related to the contribution of the quasi-periodic waves on the microscale.

\subsection*{Notations}
We end the introduction with some words on the  notation used in this article. 

Vectors and vector-valued functions are represented by lower-case boldface symbols with the exception of the co-ordinate points.  $\{ \bfe_1, \bfe_2, \bfe_3 \}$ denotes the Euclidean basis in $\RR^3$. For a vector $\bfu \in \RR^3$, we denote by $u_i$ its component with respect to $\bfe_i$, and write
$$
\bfu = (u_1,u_2,u_3) = \sum_{i=1}^3 u_i \bfe_i.
$$

Points in $\RR^3$ will be denoted by the symbol $x$ and points in the unit cell $\square := [0,1)^3$ will be denoted by $y$. The notation $\partial_{i}$ will be used to denote partial differentiation with respect to the $i$-th coordinate variable, and we shall replace the suffix $i$ with $x_i$ or $y_i$ when we wish to emphasis the macroscopic or microscopic variable. Similarly, the notion ${\rm div}_x $, ${\rm div}_y$, $\bfnabla_x$, and $\bfnabla_y$, are used for the divergence  or gradient of a  function in terms of $x$ or $y$. 

Throughout $\Omega$ is a domain in $\RR^d$, $d \ge 1$, $\square : = (0,1)^d$ and ${\bftheta} \in [0,2\pi)^d$. All of the functions, even if real-valued, are considered to take values in the complex field.

 The space  $C^\infty_{\#}(\square)$ denotes the usual space of smooth $\square$-periodic functions.  Whereas, $C^\infty_{\bftheta}(\square)$ shall denote the space of smooth functions $\varphi(y)$ whose functions and derivatives are ${\bftheta}$-quasi-periodic  with respect to $y$: $\varphi(y + \bfe_j) = \exp(\i {\bftheta}_j) \varphi(y)$ for each $y \in \square$ and each Euclidean basis vector $\bfe_j$, $j =1,\ldots,d$.  Equivalently,
	\[
	C^\infty_{\bftheta}(\square) = \{ \phi  \, | \, \phi = e^{\i \bftheta \cdot y} \psi , \psi \in C^\infty_{\#}(\square) \}.
	\]
Note that $C^\infty_{\bfzero}(\square) = C^\infty_{\#}(\square)$ and use the latter to avoid confusion with the notation for the space compactly supported smooth functions.

The Sobolev space $H^1_{\#}(\square)$ is the usual Sobolev space of $H^1$ $\square$-periodic functions. Whereas $H^1_\bftheta(\square)$
is defined as the closure of $C^\infty_\bftheta(\square)$ with respect to the $H^1$ norm,
or equivalently as
\begin{equation}
\label{h1thetadef}
H^1_\bftheta(\square) : = \{ e^{\i \bftheta \cdot y} u_{\#} \, | \, u_{\#} \in H^1_{\#}(\square) \}.
\end{equation}
Also, we note $H^1_{\bfzero}(\square) = H^1_{\#}(\square)$ and,  in this situation, we use the latter to avoid confusion with the Sobolev space of zero trace $H^1$ functions.

For subsets $\{ S_\ep\}_{\ep}$ and $S$ of $\RR^d$ we say that $S_\ep$ converges to $S$ in the Hausdorff sense if the following conditions hold:
\begin{enumerate}
	\item{ For every $\lambda_\ep \in S_\ep$ such that $\lambda_\ep$ converges to some $\lambda_0$, then   $\lambda_0 \in S$.
	}
	\item{ For every $\lambda_0 \in S$ there exists$\lambda_\ep \in S_\ep$ such that $\lim_\ep \lambda_\ep = \lambda_0$.
	}
\end{enumerate}
We shall use the notation $\lim_\ep S_\ep = S$ when a sequence of sets $S_\ep$ Hausdorff converges to $S$.

The Einstein summation convention will not be used in this article, that is we do not sum with respect to repeated indices.
\section{Quasi-periodic two-scale convergence}
\label{sec:qpcon}
In this section we introduce an appropriate notion of convergence that will account for the presence of microscopic oscillations that are quasi-periodic in nature. This convergence will turn out to be a natural extension of the standard (periodic) two-scale convergence introduced by G. Nguetseng\cite{Ng}-G. Allaire \cite{All}. In particular, we aim to use this extended notion of two-scale convergence to study the spectral properties of operator families in homogenisation theory in a similar vein to that first introduced by V. V. Zhikov \cite{Zh1,Zh2}.

\subsection{Motivation}\label{mot}
\hspace{10pt}
We shall motivate the notion of quasi-periodic two-scale convergence here. This motivation is based on the principle goal of 
 characterising the spectral asymptotics of high-contrast elliptic operators.

\hspace{10pt}Let $\ep$ be a sequence of positive real numbers with limit zero. Consider the differential operator $A_\ep : \mathcal{D}(A_\ep) \subset L^2(\mathbb{R}^d) \rightarrow L^2(\mathbb{R}^d)$ whose action is given by
\begin{equation*}
A_\ep u: = -{\rm div} \big( a_\ep( \tfrac{\cdot}{\ep}) \bfnabla u  \big)  
\end{equation*}
and domain $\mathcal{D}(A_\ep)$ consists of $u$ for which $A_\ep u \in L^2(\mathbb{R}^d)$. Here $a_\ep$ are $\square$-periodic measurable functions,
 that are bounded and elliptic: $\exists \nu_1, \nu_2 >0$ such that
\[
\nu_1 I \le \ a_\ep \le \nu_2 I.
\]
In this article, we focus on $a_\ep$ that are uniformly bounded, i.e. $\nu_2$ is independent of $\ep$, but $a_\ep$ may degenerate in the sense that $\nu_1 = \nu_1(\ep)$ with $\lim_\ep \nu_1(\ep) \ge 0$. We are interested in analysing the structure of the spectrum $\sigma(A_\ep)$ of $A_\ep$ in the limit of $\ep$. 
The strategy of the study is to establish the existence of some operator $A_0$ such that $\sigma(A_\ep)$ Hausdorff converges to $\sigma(A_0)$; i.e. the following conditions hold:
\begin{enumerate}
	\item{ For every $\lambda_\ep \in \sigma(A_\ep)$ such that $\lambda_\ep$ converges to some $\lambda_0$ we deduce that $\lambda_0 \in \sigma(A_0)$.
	}
\item{ For every $\lambda_0 \in \sigma(A_0)$ we find $\lambda_\ep \in \sigma(A_\ep)$ such that $\lim_\ep \lambda_\ep = \lambda_0$.
}
\end{enumerate}

A crucial question is how to determine the operator $A_0$. For example, in classical and semi-classical high-contrast problems, $A_0$ turns out the be the strong two-scale resolvent homogenised limit of $A_\ep$, cf. \cite{Zh1,Zh2,Co13}.  To develop some intuition on what to expect in the general case, let us recall an important result from the spectral theory of elliptic  operators with $\ep \square$-periodic coefficients: the Floquet-Bloch decomposition (see for example \cite{Ku} for more details). This result states that the following characterisation of $\sigma(A_\ep)$ holds:
\[
\sigma(A_\ep) = \bigcup_{\bfTheta \in \left[0,\tfrac{2\pi}{\ep}\right)^d} \sigma(A_\ep(\bfTheta))
\]
where $A_\ep(\bfTheta) : \mathcal{D}(A_\ep(\bfTheta)) \subset L^2(\ep\square) \rightarrow L^2(\ep\square)$, describe a  family of densely defined self-adjoint operators with compact resolvent given by the action that $A_\ep(\bfTheta) u = f \in L^2(\square)$ if $u \in H^1_{\#}(\ep\square)$ solves
\begin{equation*}
\begin{aligned}
-{\rm div} \big( a_\ep( \tfrac{x}{\ep}) \bfnabla   e^{\i \bfTheta \cdot x } u  \big)  &=  e^{\i \bfTheta \cdot x } f(x), \quad x\in \ep \square. 
\end{aligned} 
\end{equation*}
Taking the above into consideration we see that $\lambda_\ep \in \sigma(A_\ep)$ if, and only if, there exists $\bfTheta_\ep \in  \left[0,\tfrac{2\pi}{\ep}\right)^d$ and non-trivial $u_\ep \in H^1_{\#}(\ep \square)$ such that
\begin{equation*}
 \begin{aligned}
-{\rm div} \big( a_\ep( \tfrac{x}{\ep}) \bfnabla e^{\i \bfTheta \cdot x }u_\ep  \big)  &= \lambda_\ep e^{\i \bfTheta \cdot x } u_\ep(x), \quad x\in \ep \square.
\end{aligned} 
\end{equation*}

By a change of variables $y = x / \ep$ and $\bftheta =  \ep \bfTheta$, we see that $w_\ep(y) : = e^{\i {\bftheta} \cdot  y} u_\ep (\ep y)$ solves
\begin{equation}\label{blochmotp}
 \begin{aligned}
-{\rm div} \big( \ep^{-2} a_\ep( y) \bfnabla w_\ep  \big)  &= \lambda_\ep w_\ep(y), \quad y\in \square, 
\end{aligned} 
\end{equation}
and $w_\ep$ belongs to the space of $H^1(\square)$ functions that satisfy the condition
\[
w_\ep(y + z) = e^{\i \bftheta \cdot z} w_\ep(y), \qquad y\in  \square, \ z \in \ZZ^d,
\]
{for some $\bftheta \in [0,2\pi)^d$}. This condition is typically referred to as the Bloch or quasi-periodic condition and $\bftheta$ is known as the quasi-momentum. Note that $\bftheta = \bf0$ is the usual periodicity condition. 
The Sobolev space of $H^1(\square)$ $\bftheta$-quasi-periodic functions
coincides with  $H^1_\bftheta(\square)$ which, we recall, 
is be defined as
	\begin{equation*}
	H^1_\bftheta(\square) : = \{ e^{\i \bftheta \cdot y} u_{\#} \, | \, u_{\#} \in H^1_{\#}(\square) \}.
	\end{equation*}

 The general principle to observe here is that if we wish to study the asymptotic behaviour of the spectrum $\sigma(A_\ep)$ we need to keep track of the  eigenfunctions that are  $\bftheta$-quasi-periodic on micro-scale $y := x / \ep$, {\it for all} $\bftheta \in [0,2\pi)^d$. The notion of quasi-periodic two-scale convergence, introduced in Section \ref{s:qpcon2} below,  performs such a task. 

\hspace{10pt} We note here that in the case of the whole space, as discussed above, one need not refer to a notion of quasi-periodic two-scale convergence to study the asymptotics of the spectrum; one may study the norm-resolvent limits of the operators $A_\ep(\theta)$ to study spectral asymptotics, cf. \cite{HeLi} where the point-wise (in $\theta$) limits or \cite{ChCo} where the uniform limits were considered in the double-porosity setting. That being said, for boundary-value problems, the Bloch decomposition does not hold; nevertheless, the whole space (or Bloch) spectrum is expected to contribute asymptotically to the bounded domain spectrum  and the precluding discussion is still relevant. It is this setting that the quasi-periodic two-scale convergence will be particularly useful.

\hspace{10pt}Finally, we comment that the above discussion leads to the natural question: why in previous cases considered was it sufficient to consider the standard (periodic) two-scale limit of $A_\ep$ to ensure spectral convergence? Or put another way, when in the asymptotic limit of $\ep$ do we need consider {\it all} quasi-periodicity and not just $\bftheta = \bf0$. This shall be explained in Section \ref{sec:comparison}.

\subsection{Definition and basic properties}\label{s:qpcon2}
This section is dedicated to the introduction of the notion quasi-periodic two-scale convergence and an exposition of results that are 
 appropriate to homogenisation theory.

Recall $C^\infty_{\#}(\square)$ denotes the usual space of smooth $\square$-periodic functions and
\[
C^\infty_{\bftheta}(\square) = \{ \phi  \, | \, \phi = e^{\i \bftheta \cdot y} \psi , \psi \in C^\infty_{\#}(\square) \}.
\] The following {\bf mean-value property} will be important:
 For every $\varphi \in C^\infty_{\bftheta}(\square)$ and every $\phi \in C^\infty_0(\Omega)$  the following convergence
\begin{equation}
\label{mvprop}
\lim_{\ep \rightarrow 0}\int_\Omega  \left\vert \phi(x) \varphi\left( \tfrac{x}{\ep}\right)  \right\vert^2  \ \mathrm{d}x  = \int_{\Omega}\int_{\square}{ \left\vert \phi(x) \varphi( y) \right\vert^2 }, 
\end{equation}
holds. 
This fact follows  by noting that the assertion holds for elements in $C^\infty_{\#}(\square)$, see for example \cite[Lemma 1.3]{All}, and observing that multiplication by $\exp(-\i {\bftheta} \cdot y)$ defines an isomorphism between $C^\infty_{\bftheta}(\square)$ and $C^\infty_{\#}(\square)$ that preserves absolute value. Indeed, $\varphi$ belongs to $C^\infty_{\bftheta}(\square)$ if, and only if, $\exp(-\i {\bftheta} \cdot y) \varphi$ belongs to $C^\infty_{\#}(\square)$ and $| \varphi | = | \exp(-\i {\bftheta} \cdot y) \varphi | $ in $\square$.

 We remark here that because $C^\infty_\bftheta(\square)$ is isomorphic to $C^\infty_\#(\square) = C^\infty_{\bfzero}(\square)$, with isomorphism ${\rm exp}\big(-\i \bftheta \cdot y\big)$, then the results presented in this section\footnote{The results in this section can be established by first principles making no reference to such an isomorphism.} are immediately established for each $\bftheta \in [0,2\pi)^3$ if proved for $\bftheta = \bfzero$. We shall demonstrate this with the first result of the section and omit the remaining proofs which follow in a similar manner.
\begin{definition}
	\label{def:qp}
	Let $u_\ep\in L^2(\Omega)$ be a bounded sequence and $u \in L^2(\Omega \times \square)$. Then, we say $u_\ep$ (weakly) ${\bftheta}$-quasi-periodic two-scale converges to $u$, denoted by $u_\ep \qp u$, if the following convergence
	\begin{equation}
	\label{def:qp}
	\begin{aligned}
	\lim_{\ep \rightarrow 0}\int_\Omega u_\ep(x) \overline{\phi(x)\varphi\left( \tfrac{x}{\ep} \right)} \, \mathrm{d}x =\int_{\Omega}\int_{\square}u(x,y) \overline{\phi(x) \varphi(y)} \,\mathrm{d}y \mathrm{d}x,  \quad \forall \phi \in C^\infty_0(\Omega), \ \forall \varphi \in C^\infty_{\bftheta}(\square)
	\end{aligned}
	\end{equation}
	holds.
\end{definition}
\begin{remark}Notice that for ${\bftheta} = 0$, this is the standard notion of two-scale convergence. 
\end{remark}
The next important result states that bounded sequences in $L^2(\Omega)$ are relatively compact with respect to quasi-periodic two-scale convergence.
\begin{proposition}
	\label{prop:qpcompact}
	If $u_\ep$ is bounded in $L^2(\Omega)$ then, up to a subsequence, $u_\ep$ weakly ${\bftheta}$-quasi-periodic two-scale converges to some $u \in L^2(\Omega \times \square)$. 
\end{proposition}
\begin{proof}
	The result has been established previously for the case $\bftheta= \bfzero$, see for example \cite{Ng,All,Zh1}. Let us consider $\bftheta \neq \bfzero$.
	Note that the function $\widetilde{u}_\ep = {\rm exp} \big( - \i \bftheta \cdot y \big) u_\ep$ is bounded in $L^2$ and therefore by the assertion for $\bftheta = \bfzero$, up to a discarded subsequence, $\widetilde{u}_\ep$ ($\bfzero$-quasi-periodically) two-scale converges to some $\widetilde{u} \in L^2(\Omega \times Q)$. Now, the result follows from this fact and noting that for fixed $\varphi \in C^\infty_\bftheta(\square)$ one has 
	$$
	\begin{aligned}
	\int_\Omega u_\ep(x) \overline{ \phi(x) \varphi(\xoe)} \, {\rm d}x & =  \int_\Omega \widetilde{u}_\ep(x) \overline{ \phi(x)  {\rm exp}\big( - \i \bftheta \cdot \xoe \big)\varphi(\xoe)} \, {\rm d}x, \\[5pt]
	\int_\Omega\int_\square  {\rm exp} \big(  \i \bftheta \cdot y \big) \widetilde{u}(x,y) \overline{ \phi(x) \varphi(y)} \, {\rm d}y{\rm d}x  & = \int_\Omega\int_Q \widetilde{u}(x,y) \overline{ \phi(x)  {\rm exp} \big( - \i \bftheta \cdot y \big) \varphi(y)} \, {\rm d}y{\rm d}x,
	\end{aligned}
	$$
	and that $ {\rm exp} \big( - \i \bftheta \cdot y \big)$ is a smooth periodic function. Hence $u_\ep \qp {\rm exp} \big(  \i \bftheta \cdot y \big)\widetilde{u}$.
\end{proof}
An important result from the point of view of homogenisation theory is that the test functions $\varphi$ in \eqref{def:qp} can be taken to be quasi-periodic elements of $L^2(\square)$, i.e. the following result holds.
\begin{proposition}
	\label{prop:qptests}
	If $u_\ep \in L^2(\Omega)$ ${\bftheta}$-quasi-periodic two-scale converges to $u \in L^2(\Omega \times \square)$, then the following convergence
	$$
	\begin{aligned}
	\lim_{\ep \rightarrow 0}\int_\Omega u_\ep(x) \overline{\phi(x)\psi\left( \tfrac{x}{\ep} \right)} \, \mathrm{d}x =\int_{\Omega}\int_{\square}u(x,y) &\overline{\phi(x) \psi(y)} \,\mathrm{d}y \mathrm{d}x
	\end{aligned}
	$$
	holds for  all $\phi \in C^\infty_0(\Omega)$,  and for all $\psi \in L^2(\square)$ such that $\psi(y+\bfe_j) = \exp(\i {\bftheta}_j) \psi(y)$ for almost every $y \in \square$ and $j=1,\ldots,d$.
\end{proposition}
\begin{remark}
	If $\Omega$ is a bounded domain, as in this article, then additionally the test functions $\phi$ can be taken to be elements of $C(\overline{\Omega})$.
\end{remark}
The following results are of interest.
\begin{proposition} \label{profweakqp}{\ }
	
	\begin{enumerate}
		\item{For $u_\ep\in L^2(\Omega)$ ${\bftheta}$-quasi-periodic two-scale converging to $u \in L^2(\Omega\times \square)$ one has that
			$$
			\exp( -\i {\bftheta} \cdot \tfrac{x}{\ep}) u_\ep(x) \rightharpoonup \int_{\square} \exp(-\i {\bftheta} \cdot y) u(x,y) \, {\rm d}y \quad \text{weakly in $L^2(\Omega)$.}
			$$}
		\item{For $u_\ep\in L^2(\Omega)$ ${\bftheta}$-quasi-periodic two-scale converging to $u \in L^2(\Omega\times \square)$ then
			$$
			\liminf_{\ep \rightarrow 0 } \int_\Omega | u_\ep(x)|^2 \, {\rm d}x  \ge\int_\Omega \int_\square | u(x,y)|^2 \, {\rm d}y{\rm d}x.
			$$}
	\end{enumerate}
\end{proposition}

 A result of particular interest in high-contrast homogenisation problems is the following.
\begin{proposition}
	\label{prop:hc}
	Let $u_\ep \in H^1(\Omega)$ satisfy
	$$
	\begin{aligned}
	\sup_{\ep}|| u_\ep ||_{L^2(\Omega)} < \infty, \quad & \quad \sup_{\ep}|| \ep \bfnabla u_\ep ||_{L^2(\Omega)} < \infty.
	\end{aligned}
	$$
	Then, there exists $u \in L^2(\Omega;H^1_\bftheta(\square))$ such that, up to a subsequence, the following convergences hold:
	$$
	\begin{aligned}
	u_\ep \qp u, \quad & \quad \ep \bfnabla u_\ep \qp \bfnabla_y u.
	\end{aligned}
	$$Recall here that $H^1_\bftheta(\square)$ is given by \eqref{h1thetadef}.
\end{proposition}

\begin{proof}
	Let	$\phi$ and $\bfPsi$ denote respectively fixed arbitrary elements of $C^\infty_0(\Omega)$ and $C^\infty_{\bftheta}(\square;\CC^d)$.
	By Proposition \ref{prop:qpcompact}, there exists $u \in L^2(\Omega\times \square)$ and $\bfchi \in L^2(\Omega\times \square ; \CC^d)$  such that, up to a discarded subsequence, the following convergences hold:
	\begin{align}
	u_\ep  \qp u, \quad & \quad 	\ep \bfnabla u_\ep  \qp \chi.  \label{qplimits2}
	\end{align}{ \ }
	\hspace{10pt}  
	Note that, since $u_\ep$ is bounded in $L^2(\Omega)$, then $\ep u_\ep$ strongly converges to zero in $L^2(\Omega)$ and from  Proposition \ref{profweakqp} part 2. we conclude that
	\begin{equation}
	\label{qplimits2.1}
	\ep u_\ep \qp 0.
	\end{equation}	Let us prove that $u \in L^2(\Omega ; H^1_{\bftheta}(Q))$.  Using the convergences \eqref{qplimits2} and \eqref{qplimits2.1} we  pass to the limit in the identity
	$$
	\begin{aligned}
	 \int_{\Omega} \ep \bfnabla u_\ep(x) \cdot \overline{ \phi(x)\bfPsi(\xoe)} \, {\rm d}x & = - \int_{\Omega} u_\ep(x) \overline{ \ep {\rm div} \big( \phi(x)\bfPsi(\xoe) \big)} \, {\rm d}x  \\
	& = - \int_{\Omega} u_\ep(x)  \overline{ \ep  \bfnabla_x  \phi(x) \cdot \bfPsi(\xoe)  }\, {\rm d}x - \int_{\Omega} u_\ep(x)  \overline{  \phi(x){\rm div}_y \bfPsi(\xoe)} \, {\rm d}x,
	\end{aligned}
	$$
	 to deduce that
	$$
	\int_{\Omega}\int_\square   \bfchi(x,y) \cdot \overline{ \phi(x)\bfPsi(y)} \, {\rm d} y{\rm d}x  = - \int_{\Omega}\int_\square   u(x,y) \overline{ \phi(x) {\rm div}_y\bfPsi(y)} \, {\rm d} y{\rm d}x.
	$$
	Therefore, for almost every $x$, the functions $\bfchi(x,\cdot)$ and $u(x,\cdot)$ are related by the identity
	$$
	\int_\square   \bfchi(x,y) \cdot \overline{ \bfPsi(y)} \, {\rm d} y  = - \int_\square  u(x,y) \overline{  {\rm div}_y\bfPsi(y)} \, {\rm d}y, \qquad \forall \bfPsi\, \in C^\infty_{\bftheta}(\square;\CC^d).
	$$
 It is clear that $C^\infty_0\big( (0,1)^d\big) \subset C^\infty_{\bftheta}(\square)$ and so $u\in H^1(\square)$ with $\bfnabla_y u = \bfchi$. It remains to show $u$ belongs to $H^1_\bftheta(\square)$. This follows from noting that after performing  integration by parts in the above identity we arrive at 
\[
\int_{\partial \square} u(x,y) \overline{\bfPsi(y) \cdot \bfnu }\ {\rm d}S(y) = 0, \qquad \forall  \bfPsi\, \in C^\infty_{\bftheta}(\square;\CC^d).
\]
Setting $\bfPsi = e^{\i \bftheta \cdot y} \bfPsi_{\#}$ above, for arbitrary smooth $\square$-periodic $\bfPsi_{\#}$, demonstrates that $u_{\#}(x,\cdot) : = e^{ - \i \bftheta\cdot y} u(x,\cdot)$ is an element of $H^1(\square)$ that satisfies  periodic boundary conditions with respect to $y$. That is, $u_{\#}(x,\cdot) \in H^1_{\#}(\square)$ and so (see definition \eqref{h1thetadef}) $u(x, \cdot) \in H^1_\bftheta(\square)$.
\end{proof}
We end this section with a result that is illuminating when it comes to studying the convergence of spectra for parameter-dependent operator families. It readily provides a one-sided justification for the Hausdorff convergence of the high-contrast spectra 
to  the spectrum associated to quasi-periodic two-scale limits. This result is based on the following definition, which extends the notion of strong resolvent two-scale convergence first introduced by V. V. Zhikov in \cite{Zh1,Zh2}.

\begin{definition}\label{stopcon}
	Fix ${\bftheta} \in [0,2\pi)^d$, and let $A_\ep$ and $A$ be non-negative self-adjoint operators respectively defined in $L^2(\Omega)$ and $H$ a closed subset of $L^2(\Omega\times\square)$. We say that $A_\ep$ strong resolvent ${\bftheta}$-quasi-periodic two-scale  converges to $A$ if for every $f_\ep(x) \in L^2(\Omega)$ that ${\bftheta}$-quasi-periodic two-scale converges to $f(x,y) \in L^2(\Omega\times \square)$, the following convergence
	$$
	u_\ep = (A_\ep + I)^{-1} f_\ep \qp u = (A + I)^{-1} P f, \ \text{as $\ep \rightarrow 0$}
	$$
	holds. Here, $P$ is the orthogonal projection onto $H$ in $L^2(\Omega \times \square)$.
\end{definition}
Here, we state an important consequence of such resolvent convergence. The proof, omitted here, follows standard spectral theoretic arguments, see for example \cite{Zh1}.
\begin{proposition}
	\label{tsrescon}
	If $A_\ep$ strong resolvent ${\bftheta}$-quasi-periodic two-scale converges to $A$ then the spectrum $\sigma(A)$ of $A$ is related to the spectrum $\sigma(A_\ep)$ of $A_\ep$ in the following sense:{\ }
	
	{\ }
	
	For every $\lambda \in \sigma(A)$ there exists $\lambda_\ep \in \sigma(A_\ep)$ such that $\lambda_\ep$ converges to $\lambda$ as $\ep$ tends to zero.
	
\end{proposition}


\subsection{On the relevance of  quasi-periodic two-scale convergence in spectral asymptotics}
\label{sec:comparison}
Proposition \ref{tsrescon} informs us that, in principle,  one should consider all strong quasi-periodic two-scale limits of an operator $A_\ep$ to fully characterise its limit spectrum (in the Hausdorff sense). Yet, clearly this is not always the case: such a notion of convergence has not appeared previously, nor was it  needed, to study the spectral asymptotics of classical and particular double-porosity operators. The reason for this shall be elucidated here. Moreover, at the end of this section we shall argue when quasi-periodic convergence is necessary via a model problem that  we later study in detail  in this article. 

\subsubsection{Classical homogenisation}
Consider the resolvent problem: For fixed $f \in L^2(\Omega)$ find $u_\ep \in H^1_0(\Omega)$ such that\begin{equation}
\label{clasprob}
- {\rm div} \big( a(\xoe) \bfnabla u_\ep\big)+ u_\ep = f,
\end{equation}
where  the symmetric matrix-valued function $a$ is $\square$-periodic, elliptic and bounded: $\exists \nu >0$ such that 
\[
\nu |\xi|^2 \le a(y) \xi \cdot \overline{\xi} \le \nu^{-1} | \xi|^2, \qquad \forall \xi \in \CC^d, \, \text{ a.e. } y \in \square.
\]
The following homogenisation theorem is classical.
\begin{theorem}[Classical homogenisation theorem]
\label{chomthm}
Let $\ep$ be a sequence with limit $0$, and  $f_\ep \in L^2(\Omega)$  a sequence such that $f_\ep$ weakly converges in $L^2(\Omega)$ to some $f_0$ as $\ep$ tends to zero. Then $u_\ep \in H^1_0(\Omega)$ the solution to \eqref{clasprob}, for $f = f_\ep$,
converges weakly in $H^1_0(\Omega)$ (and strongly in $L^2(\Omega)$) to $u_0 \in H^1_0(\Omega)$ the solution to
\[
 - {\rm div} \big( a^{\rm hom} \bfnabla u_0\big) + u_0 =f_0.
\]
Here $a^{\rm hom}$ is the constant symmetric homogenised matrix determined by $a$:
\[
a^{\rm hom} \xi \cdot \xi : = \min_{N \in H^1_{\#}(\square)
} \int_\square a \big( \nabla N + \xi) \cdot \big( \nabla N + \xi), \qquad \forall \xi \in \RR^d.
\] 
\end{theorem}
It is well-known that the homogenisation theorem implies  the Hausdorff convergence of spectra (cf. \cite[Section 2]{AlCo}): 
\[
\lim_\ep \sigma(A_\ep) = \sigma(A^{\rm hom}).
\]

Let us study the quasi-periodic two-scale limits of $u_\ep$.
\begin{proposition}
\label{p:clasqp}
Fix $\bftheta \in (0,2\pi)^d$ and consider $f_\ep \in L^2(\Omega)$ such that $f_\ep \qp f_0$. Then, $u_\ep \in H^1_0(\Omega)$ $\bftheta$-quasi-periodic two-scale converges to zero; that is $u_\ep \qp 0$.
\end{proposition}
\begin{remark}
\begin{enumerate}
	\item{ Proposition \ref{p:clasqp} informs us that for the classical resolvent problem \eqref{clasprob}, the non-zero quasi-periodic micro-oscillations at leading order do not contribute to the spectral asymptotics. So one need only study the $\bftheta = 0$ quasi-periodic oscillations, i.e. the standard two-scale limit. It is well-known that the (periodic) two-scale limit coincides with the classical limit provided by Theorem \ref{chomthm}, see for example \cite{All,Zh1}.
	}
\item{ The part of the spectrum corresponding the $\bftheta$-quasi-periodic micro-oscillations, for $\bftheta \neq \bf0$, actually resides in an $\ep^{-2}$ neighbourhood of infinity; this can be formally seen from the considerations of Section \ref{mot}: for $a_\ep$ independent of $\ep$, the eigenvalues $\lambda_\ep$ in \eqref{blochmotp} are clearly of the order $\ep^{-2}$. To study such `high-frequency' spectrum one can consider the re-scaled operator $\ep^{2} A_\ep$, that is consider coefficients of the form $a_\ep = \ep^{2} a$. The precise study of such high-frequency spectra was performed in \cite{AlCo} for a broader class of moderately contrasting locally periodic coefficients. Therein, the authors provide a rigorous description of  the high-frequency spectral asymptotics in terms of non-trivial quasi-momenta $\bftheta$. This was done by introducing an appropriate notion of ``Bloch wave homogenisation". For the reduced setting of (globally) periodic coefficients, the Bloch-wave operator-limits determined therin can readily be shown to be equivalent to the $\bftheta$-quasi-periodic two-scale limits.
}
\end{enumerate}
\end{remark}
\begin{proof}[Proof of Proposition \ref{p:clasqp}]
The sequence $f_\ep$ weakly converges, cf. Proposition \ref{profweakqp} part 1., and so is bounded. Multiplying \eqref{clasprob} (for $f = f_\ep$) and integrating over $\Omega$, and using the ellipticity of $a$, produces the a-priori bound\[
\Vert u_\ep \Vert_{L^2(\Omega)}^2 + \nu \Vert \bfnabla  u_\ep \Vert_{L^2(\Omega)} ^2 \le \Vert f_\ep \Vert_{L^2(\Omega)}^2 \le  C < \infty.
\]
Applying Proposition \ref{prop:hc}, we deduce that there exists $u \in L^2(\Omega;H^1_\bftheta(\square))$ such that, up to a subsequence, the following convergences hold:
$$
\begin{aligned}
u_\ep \qp u, \quad & \quad \ep \bfnabla u_\ep \qp \bfnabla_y u.
\end{aligned}
$$
Let us show $u=0$:  $\bfnabla u_\ep$ is a bounded sequence and so $\ep \bfnabla u_\ep$ strongly converges to zero in $L^2$. Therefore, by Proposition \ref{profweakqp} part 2., we deduce $\bfnabla_y u = 0$. As $\square $ is connected it follows that $u$ is constant. Yet $u_\ep \in H^1_\bftheta(\square)$ and there are no non-trivial constant $\bftheta$-quasi-periodic functions for $\bftheta \neq 0$, see \eqref{h1thetadef}. Hence, $u = 0$.
\end{proof}

\subsubsection{Double-porosity model}\label{dpexamplesec}
Consider the resolvent problem:   For fixed $f \in L^2(\Omega)$ find $u_\ep \in H^1_0(\Omega)$ such that\begin{equation}
\label{h.cprob}
- {\rm div} \big( a_\ep(\xoe) \bfnabla u_\ep\big)+ u_\ep = f.
\end{equation}
Here 
\begin{equation*}
a_\ep(y) = \left\{ \begin{matrix}
a_{1}(y), & y \in Q_1, \\
\ep^2 a_{0}(y), & y \in Q_0,
\end{matrix} \right.
\end{equation*}
where $Q_0$ is a smooth compactly contained subset of $\square$ such that, for $Q_1 : = \square \backslash \overline{Q_0}$, the periodic extension
\[
F_1 : = \bigcup_{z\in \ZZ^d} (Q_1 +z)
\] forms a connected set in $\RR^d$. The functions $a_i$, $i =0 ,1$ are taken to be real-valued, elliptic and bounded on $Q_i$.  The following homogenisation theorem is established in \cite[Theorem 5.1]{Zh1}.
\begin{theorem}
\label{dphomthm}
Suppose $f = f_\ep$ in the right-hand side of \eqref{h.cprob} two-scale converges to some $f_0$, that is $f_\ep \qp f_0$ for $\bftheta = 0$. Then, the sequence of solutions $u_\ep$ two-scale converges to $u_0(x,y) = u(x) + v(x,y)$, where $(u,v)$ belongs to
\[
V_0 = H^1_0(\Omega) \oplus L^2(\Omega;H^1_0(Q_0))
\] and uniquely solves 
\begin{equation}\label{dplim}
\begin{aligned}
\int_\Omega &  a^{\rm hom}_{dp} \bfnabla_x u(x) \cdot \overline{\bfnabla_x \phi(x)}\, {\rm d}x + \int_\Omega \int_\square a_0(y)\bfnabla_y v(x,y) \cdot \overline{\bfnabla_y \varphi(x,y)} \, {\rm d}y{\rm d}x   \\ &+ \int_\Omega \int_\square (u(x) + v(x,y)) \cdot \overline{ (\phi(x) + \varphi(x,y))} \, {\rm d}y{\rm d}x =   \int_\Omega \int_\square f_0(x,y) \cdot \overline{ (\phi(x) + \varphi(x,y))} \, {\rm d}y{\rm d}x, \\ & \hspace{.5\textwidth}\qquad \forall \phi \in H^1_0(\Omega), \, \forall \varphi \in L^2(\Omega;H^1_0(Q_0)).
\end{aligned}
\end{equation}
Here, $a^{\rm hom}_{dp}$ is the constant symmetric and positive homogenised matrix for perforated domains determined by $a_1$:
\[
a^{\rm hom}_{\rm dp} \xi \cdot \xi : = \min_{N \in H^1_{\#}(Q_1)
} \int_{Q_1} a_1 \big( \nabla N + \xi) \cdot \big( \nabla N + \xi), \qquad \forall \xi \in \RR^d.
\] 
\end{theorem}
This result informs us that $A_\ep$ strongly two-scale converges to the operator $A_{\bfzero}$, defined in $L^2(\Omega \times \square)$, associated to the above two-scale limit resolvent problem. Therefore, by Proposition \ref{tsrescon} for $\bftheta = \bfzero$, the lower-semicontinuity of the spectral convergence is ensured. In fact, Zhikov proved in \cite[Theorem 8.1]{Zh1}, under the condition that $F_1$ is connected in $\RR^d$, the stronger result 
\[
\lim_\ep \sigma(A_\ep) = \sigma(A_{\bfzero}).
\]

Let us determine the strong resolvent quasi-periodic two-scale limits of $A_\ep$. 
\begin{proposition}
\label{thetaqphclim}
Fix $\bftheta \in (0,2\pi)^d$. Suppose $f = f_\ep$ in the right-hand side of \eqref{h.cprob} such that $f_\ep \qp f_0$ to some $f_0 \in L^2(\Omega \times \square)$. Then, the sequence of solutions $u_\ep$ $\bftheta$-quasi-periodically two-scale converges to $v_0(x,y) \in L^2(\Omega;H^1_0(Q_0))$ the  solution to 
\begin{equation}\label{hcqplim}
\begin{aligned}
 \int_\Omega \int_\square a_0(y) \bfnabla_y v_0(x,y) \cdot \overline{\bfnabla_y \varphi(x,y)} \, {\rm d}y{\rm d}x   + \int_\Omega \int_\square  v(x,y) \cdot \overline{ \varphi(x,y)} \, {\rm d}y{\rm d}x =   \int_\Omega \int_\square f_0(x,y) \cdot \overline{  \varphi(x,y)} \, {\rm d}y{\rm d}x, \\  \hspace{.5\textwidth}\qquad  \forall \varphi \in L^2(\Omega;H^1_0(Q_0)).
\end{aligned}
\end{equation}
\end{proposition}

\begin{remark}\

1. Note that $A_\bftheta = \mathcal{A}$ is independent of $\bftheta$ for $\bftheta \neq  \bfzero$, and its spectrum is the point spectrum given by  the operator whose action is $u \mapsto - {\rm div} ( a_0 \bfnabla u )$ with domain $\{ u \in H^1_0(Q_0) \, | \, - {\rm div} ( a_0 \bfnabla u ) \in L^2(Q_0) \}$. 

2. It is easy to see $\mathcal{A} \subset A_{\bfzero}$ (by noting that setting $u=\phi = 0$ in \eqref{dplim} gives  \eqref{hcqplim}) and so 
\[
\bigcup_{\bftheta \neq \bfzero} \sigma(A_{\bftheta}) = \sigma(\mathcal{A}) \subset \sigma(A_{\bfzero}).
\]
The set 
\[
\bigcup_{\bftheta \neq \bfzero} \sigma(A_{\bftheta}),
\]
is the limit spectrum  arriving from quasi-periodic micro-oscillations.

3. 
The restriction of the limit spectrum $\sigma(A_{\bfzero})$ to $
\bigcup_{\bftheta \neq \bfzero} \sigma(A_{\bftheta})$  is achieved  by considering the purely macro-scopic component $u(x)$ (of eigenfunctions) to be zero. For this reason, we coin this spectrum to be pure Bloch spectrum. 
  In the simplified setting of double-porosity the pure Bloch  spectrum is point spectrum (due to the fact $A_\bftheta = \mathcal{A}$ is independent of $\bftheta$ for $\bftheta \neq  \bfzero$).  In general, we expect this spectrum to have band-gap structure, and the gaps have only contracted to points here due to the geometric constraint that $F_1$ is connected in $\RR^d$. This expectation is verified in Section \ref{sec:limspec}.

3. Even though the strong resolvent $\bftheta$-quasi-periodic limit of $A_\ep$ exists, it has trivial dependence on $\bftheta$, $\bftheta \neq \bfzero$ and more importantly is a restriction of the two-scale limit $A_{\bfzero}$. Hence, one need only consider $A_{\bfzero}$, and this explains why in this setting one is to expect Zhikov's result $\lim_\ep \sigma(A_\ep) = \sigma(A_{\bfzero})$. In general, the limit $A_{\bfzero}$ will not be sufficient to capture the full spectral asymptotics.

\end{remark}

\begin{proof}[Proof of Proposition \ref{thetaqphclim}]
Let us consider $a_1$ (respect. $a_0$) to be extended by zero into $Q_0$ (respect $Q_1$),  and consider $\nu>0$ to be the constant such that
\[
a_1 + a_0 \ge \nu.
\]
The solution $u_\ep$ solves
\begin{equation}
\label{dpqpe1}
\int_\Omega (a_1(\xoe) + \ep^2 a_0(\xoe) ) \bfnabla u_\ep \cdot \overline{ \bfnabla \phi} +  \int_\Omega  u_\ep  \overline{ \phi} = \int_\Omega  f_\ep  \overline{ \phi}, \qquad \forall \phi \in H^1_0(\Omega).
\end{equation}
Setting $\phi = u_\ep$ in the above variational problem and using the fact that $a_1 + a_0 \ge\nu$, we deduce the a-priori bound (for $\ep \le 1$)
\[
\Vert u_\ep \Vert^2_{L^2(\Omega)} + \nu  \Vert \ep \bfnabla u_\ep \Vert^2_{L^2(\Omega)} \le \Vert f_\ep \Vert_{L^2(\Omega)}^2 \le C < \infty. 
\]
Additionally, we have the bound
\[
\Vert \sqrt{a_1}(\xoe) \bfnabla u_\ep \Vert_{L^2(\Omega)}^2  \le  \Vert f_\ep \Vert_{L^2(\Omega)}^2 \le C < \infty. 
\]
Indeed, $a_1 \ge 0$ and
\[
\int_\Omega a_1(\xoe) \bfnabla u_\ep \cdot \overline{ \bfnabla u_\ep} \le  \int_\Omega (a_1(\xoe) + \ep^2 a_0(\xoe) ) \bfnabla u_\ep \cdot \overline{ \bfnabla u_\ep} +  \int_\Omega  | u_\ep|^2 = \int_\Omega  f_\ep  \overline{ u_\ep}.
\]
By Proposition \ref{prop:hc} it follows that, up to a discarded subsequence,  
\begin{equation}
\label{hcqpe2}
u_\ep \qp v_0, \quad \text{ and  } \quad \ep \bfnabla u_\ep \qp \bfnabla_y v_0,
\end{equation}
for some $v_0 \in L^2(\Omega ; H^1_\bftheta(Q))$. 

Let us show that $v_0 \in L^2(\Omega;H^1_0(Q_0))$. By Proposition \ref{prop:qptests} it follows that
\[
\sqrt{a_1}(\xoe) \ep \bfnabla u_\ep \qp \sqrt{a_1}(y) \bfnabla_y v_0.
\]
Yet $\sqrt{a_1}(\xoe) \ep \bfnabla u_\ep$ strongly converges to zero in $L^2(\Omega)$. Therefore $\sqrt{a_1}(y) \bfnabla_y v_0 = 0$, which is equivalent to $ \bfnabla_y v_0 = 0$ on $Q_1$ (recall $a_1$ is positive on $Q_1$ and zero on $Q_0$). As $Q_1$ is connected then $v_0$ is constant in $Q_1$.  Now, since the periodic extension $F_1 = \bigcup_{z\in \ZZ^d} (Q_1 +z)$ forms a connected set then $v_0$ is constant in $F_1$. Yet,  $v_0 \in H^1_\bftheta(\square)$ for $\bftheta \neq \bfzero$ and consequently this constant is zero, cf \eqref{h1thetadef}.

It remains to prove $v_0$ solves \eqref{hcqplim}. This can easily be deduced by passing the the $\bftheta$-quasi-periodic limit in \eqref{dpqpe1} for test functions $\phi(x) = \psi(x)\varphi(\xoe)$, $\psi \in H^1_0(\Omega)$, $\varphi \in H^1_0(Q_0)$ and using convergences \eqref{hcqpe2}.
\end{proof}

\subsubsection{An example with non-trivial quasi-periodic limits}
Let us provide an example which demonstrates that in general  the family $A_{\bftheta}$, of strong resolvent $\bftheta$-quasi-periodic limits to $A_\ep$, are not restrictions of $A_{\bfzero}$.

Suppose, we consider \eqref{h.cprob} for coefficients
\[
a_\ep = a_1 + \ep^2 a_0,
\]
and $a_i$ are real-valued $\square$-functions such that $a_1 \ge 0$ and $a_1 + a_0 \ge \nu > 0$.  Let $f_\ep$ be a bounded sequence and $u_\ep$ solve \eqref{h.cprob} for $f = f_\ep$. Arguing as in the proof of Proposition \ref{thetaqphclim}, we see that $u_\ep$ to solution to \eqref{h.cprob}  will satisfy the a-priori bounds
\[
\Vert u_\ep \Vert^2_{L^2(\Omega)} + \nu  \Vert \ep \bfnabla u_\ep \Vert^2_{L^2(\Omega)} + 
\Vert \sqrt{a_1}(\xoe) \bfnabla u_\ep \Vert_{L^2(\Omega)}^2  \le  \Vert f_\ep \Vert_{L^2(\Omega)}^2 \le C < \infty. 
\]
In particular, Proposition \ref{prop:hc} informs us that up to a subsequence 
	\begin{equation*}
	u_\ep \qp u_0, \quad \text{ and  } \quad \ep \bfnabla u_\ep \qp \bfnabla_y u_0,
	\end{equation*}
	for some $u_0 \in L^2(\Omega ; H^1_\bftheta(Q))$. Moreover, by an application of Proposition \ref{prop:qptests} we deduce that 
	\[
	\sqrt{a_1} \bfnabla_y u_0 = 0.
	\]
	Denoting by $V_\bftheta$ the closed linear subspace of $H^1_\bftheta(\square)$ given by\footnote{The space $V_{\bfzero}$ was first introduced in \cite{SmMoM} and coined the space of microscopic oscillations. We have appropriately extended this notion to $\bftheta$-quasi-periodic oscillations here, $\bftheta \neq \bfzero$.}
	\[
	V_\bftheta = \{  v \in H^1_\bftheta(\square) \, | \, 	\sqrt{a_1} \bfnabla_y v= 0 \}.
	\]
		Suppose we show an example where $V_{\bftheta}$ is not a subset of $V_{\bfzero}$ for some  non-trivial $V_\bftheta$, $\bftheta \neq \bfzero $.
 Then, for such examples we should not expect  that  $A_{\bfzero}$ is  an extension of $A_{\bftheta}$, nor should we expect $\sigma(A_\bftheta) \subset \sigma(A_{\bfzero})$. Let us provide such an example. The conjectures (stated immediately above) based on this example will be proved rigorously in the remainder of the article.

Suppose $Q_1$ is the cylinderical domain
\[
Q_1  := [0,1) \times [\tfrac{1}{4}, \tfrac{3}{4} ]^2,
\]
and
\[
a_1(y) : = \left\{ 
\begin{aligned}
1, & \quad & y \in Q_1, \\ 
0, & &y \in Q_0.
\end{aligned} \right.
\]
Notice that  $F_1 : = \bigcup_{z\in \ZZ^3} (Q_1 + z)$ the periodic extension of $Q_1$ into $\RR^d$ consists of infinitely many mutually disjoint cylinders  $C_{l}: = \RR  \times [\tfrac{1}{4}+ l_1, \tfrac{3}{4} + l_2]^2$, $\forall l \in \ZZ^2$. That is, the assumptions of Subsection \ref{dpexamplesec}, and in particular \cite{Zh1,Zh2}, do not hold.

Now $v \in V_\bftheta$ if, and only if, $v\in H^1_\bftheta(\square)$ with $v=c$ for some $c \in \CC$ on $Q_1$. This implies, cf \eqref{h1thetadef}, that 
\[
v(1, y_2, y_3) = e^{\i \theta_1} v(0,y_2,y_3), \quad (y_1,y_2) \in (0,1)^2, \,j\in \{1,2,3\}, 
\]
(in the sense of trace). Then, for $(y_2,y_3) \in (\tfrac{1}{4},\tfrac{3}{4})^2$ we arrive at the condition
\[
c= e^{\i \theta_1 } c. 
\]
Therefore, if $\theta_1 \neq 0$, then the above condition only holds if $c = 0$. That is $v$ must necessarily be zero on $Q_1$. On the other hand, if $\theta_1 = 0$ then any $H^1_\bftheta(\square)$ function that is constant on $Q_1$ belongs to $V_\bftheta$. In particular we see that $V_\bftheta$ does not belong to $V_{\bfzero}$  for all $\bftheta \in (0,2\pi)^3$.

 \begin{remark}
 Note that if $Q_1$ contains a connected subset which joins two opposite faces of the square $\square$ then the  space $V_{\bftheta}$ non-trivially depends on $\bftheta$. Consequently, non-trivial limit Bloch spectrum is  expected for $\bftheta$ aligned   orthogonally to these faces.
 \end{remark}

The remainder of the article is dedicated to determining the strong quasi-periodic two-scale limits of $A_\ep$ for such fibre-like inclusions. Moreover, we  demonstrate that indeed  $A_\bftheta$ are not restrictions of $A_{\bfzero}$ and that $\sigma(A_\bftheta)$ form non-trivial subsets of $\lim_\ep \sigma(A_\ep)$.
 
\section{Problem formulation and Homogenisation}
\label{sec:pfandhom}
In this article we are concerned with the asymptotic analysis of the resolvent problem 
\begin{equation}
\label{resolventprob0}
\left\{ \begin{aligned}
& \text{Find $u_\ep \in H^1_0(\Omega)$ such that} \\
& -{\rm div} \big( a_\ep( \xoe) \bfnabla u_\ep  \big) + u_\ep = f_\ep \qquad \text{in $\Omega$}
\end{aligned}  \right. \qquad
\end{equation}
where $\ep <1$ is a small parameter,  $\Omega$ is a smooth open bounded star-shaped domain\footnote{All the results and proofs follow through in an identical manner for the case where $\Omega$ is the whole space.}  and $f_\ep \in L^2(\Omega)$ known. The coefficient $a_\ep$ is given by
\begin{equation}
\label{hccoeffs}
\hspace{-6pt}a_\ep(y) = \left\{ \begin{matrix}
a_1(y), & y \in Q_1, \\
\ep^2 a_0(y), & y \in Q_0,
\end{matrix}  \right. \quad 0\le a_i, a_i^{-1} \in L^\infty(Q_i),\quad  a_i = 0 \text{ on $Q_{1-i}$},  i = 0,1,
\end{equation}
and the regions $Q_0$ and $Q_1$ are described as follows (cf. Figure \ref{fig:inclusions}).

{\bf Geometric assumptions.} For $j = 1,2,3$ we consider smooth domains $S_j$ compactly contained in $(0,1)^2$ that have mutually disjoint closures. We denote by $C_j$ the cylinder aligned to the $j$-th co-ordinate axis with cross-section $S_j$, i.e. $ C_1 : =\{ y \in (0,1)^3 \, \big| \, y \in (0,1)\times S_1 \}$, $ C_2 : =\{ y \in (0,1)^3 \, \big| \, y = (z_2,z_3,z_1), z \in (0,1)\times S_2 \}$ and $ C_3 : =\{ y \in (0,1)^3 \, \big| \, y = (z_3,z_1,z_2), z \in (0,1)\times S_3 \}$.

Then, for a given non-empty subset $\I$ of $\{1,2,3\}$, we consider $Q_1 = \cup_{i \in \I} C_i$. We denote by $\Gamma_i = \partial C_i \backslash \partial Q$ and $\Gamma = \bigcup_{i \in \I} \Gamma_i = \partial Q_1 \backslash \partial Q$.

\begin{figure}[h]
	\centering
	\includegraphics[width=\linewidth]{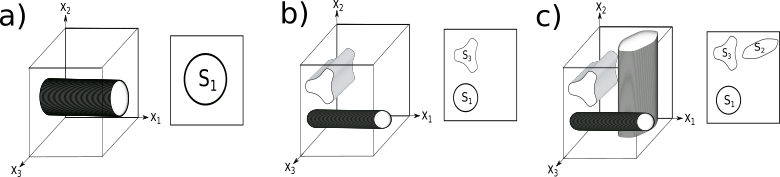}
	\caption{a) An example stiff component $Q_1$ consisting of one cylinder $C_1$ extending in the $x_1$ direction, i.e. $\I = \{1\}$. b) The stiff component $Q_1$ consists of two disjoint cylinders extending in the co-ordinate directions $x_1,$ and $x_3$, i.e. $\I = \{1,3\}$. c) $Q_1$ consists of mutually disjoint cylinders extending in all co-ordinate directions, i.e. $ \I = \{1,2,3\}$. }
	\label{fig:inclusions}
\end{figure}

Under this geometric assumption we determine for each $\bftheta \in [0,2\pi)^3$ the strong resolvent $\bftheta$-quasi-periodic two-scale limit, cf. Section \ref{sec:qpcon}, of the self-adjoint operator $A_\ep$ associated to resolvent problem \eqref{resolventprob0}. That is, for a fixed ${\bftheta} \in [0,2\pi)^3$ and a given bounded sequence $f_\ep \in L^2(\Omega)$ such that $f_\ep \qp f$, i.e.
\begin{equation*}
\begin{aligned}
\lim_{\ep \rightarrow 0}\int_\Omega f_\ep(x) \overline{\phi(x)\varphi\left( \tfrac{x}{\ep} \right)} \, \mathrm{d}x =\int_{\Omega}\int_{\square}f(x,y) \overline{\phi(x) \varphi(y)} \,\mathrm{d}y \mathrm{d}x,  \quad \forall \phi \in C^\infty_0(\Omega), \ \forall \varphi \in C^\infty_{\bftheta}(\square)
\end{aligned}
\end{equation*}
we aim to determine the $\bftheta$-quasi-periodic two-scale limit behaviour of the solution $u_\ep \in H^1_0(\Omega)$  to
\begin{equation}
\label{resolventproblem}
\begin{aligned}
\int_{\Omega} \big(a_1 (\tfrac{x}{\ep}) + \ep^2 a_0 (\tfrac{x}{\ep})\big) \bfnabla u_\ep(x) & \cdot \overline{\bfnabla \phi}(x) \, {\rm d}x \ + \int_{\Omega} u_\ep(x)  \overline{\phi}(x) \, {\rm d}x  = \int_{\Omega} f_\ep(x)  \overline{\phi}(x) \, {\rm d}x, \qquad \forall \phi \in C^\infty_0(\Omega).
\end{aligned}
\end{equation}

As $f_\ep$ is bounded in $L^2(\Omega)$, upon setting $\phi = u_\ep$ in \eqref{resolventproblem} we deduce that the sequences 
\begin{equation}\label{apbounds}
\begin{aligned}
|| \sqrt{a_1}(\tfrac{\cdot}{\ep}) \bfnabla u_\ep ||_{L^2(\Omega;\CC^3)}, \quad || \ep \bfnabla u_\ep ||_{L^2(\Omega;\CC^3)}, \ \ \text{ and} \ \ || u_\ep ||_{L^2(\Omega)},
\end{aligned}
\end{equation}
are bounded. Let us describe the $\bftheta$-quasi-periodic two-scale limit, referring to Section \ref{proofthm} for the details.

The limit of $u_\ep(x)$ will be a function  $u(x,y)$, of two variables $x \in \Omega$, $y \in Q$, that is $\bftheta$-quasi-periodic with respect to the second variable $y$, cf Proposition \ref{prop:hc}. Furthermore, due to the fact that in each cylinder $C_i$, $i \in \I$, the gradient of $u_\ep$ is bounded, the limit $u$ necessarily belongs to the (Bochner) space $L^2(\Omega; V_\bftheta)$ where
\begin{equation}
\label{Vtheta}
V_\bftheta : = \{ v \in H^1_\bftheta(Q) \, | \, \text{ $v$ is constant in $C_i$ for each $i \in \I$} \}.
\end{equation}
It follows from this (see \eqref{h1thetadef}) that $u$ is non-zero in cylinder $C_i$ if and only if the $i$-th component $\theta_i$ of $\bftheta$ is zero. If $\theta_i = 0$, then we determine that $u_i$ is not only non-trivial but it is actually more regular in the $x_i$-th coordinate direction: $\partial_{x_i} u_i \in L^2(\Omega)$. 

More precisely, for $\I^\bftheta$ the subset of indexes $\I \subseteq \{1,2,3\}$ given by $\mathcal{I}^\bftheta : = \{ i \in \I \, | \, \theta_i = 0\}$, we denote by  $\CC^\bftheta$ the closed subspace of $\CC^3$ spanned by $\{ \bfe_i \}_{i \in \I^\bftheta}$,\footnote{Note that $\CC^\bftheta$ is either the whole space, a plane or a line in $\CC^3$.} and show that the function $u$ belongs to the set 
\begin{equation}
\label{limitspace}
\begin{aligned}
U_\bftheta = &\big\{ u \in L^2(\Omega;H^1_\bftheta(Q)) \, \big| \, u = u_i \text{ on $\Omega \times C_i$},  \\ 
&\text{ for some $\bfu \in L^2(\Omega; \CC^\bftheta)$ with $\partial_i u_i \in L^2(\Omega)$ and $u_i \nu_i=0$ on $\partial \Omega$ } \big\},
\end{aligned}
\end{equation}
which is clearly a Hilbert space when endowed  with the inner product
$$
\begin{aligned}
( u,v)_{U_\bftheta} : = \sum_{i \in \I^\bftheta} \int_{\Omega}  \partial_i u_i(x) \overline{\partial_i v_i(x)}  \, {\rm d}x + \int_\Omega \int_{Q_0} \bfnabla_y u(x,y) \cdot \overline{\bfnabla_y v(x,y)} \, {\rm d}y{\rm d}x \\ + \int_\Omega \int_Q  u(x,y)\overline{v(x,y)} , {\rm d}y{\rm d}x.
\end{aligned}
$$
Here, $\bfnu$ is the outer unit normal to $\partial \Omega$.
%
%
%
%
%

For each fibre $C_i$ there corresponds an effective constant material parameter $a^{\rm hom}_i >0$ given by
\begin{equation}
\label{ahom}
a^{\rm hom}_i =\int_{C_i}  a_1(y)[\partial_{y_i} N^{(i)}(y) + 1] \, {\rm d}y,
\end{equation}
where $N^{(i)} \in H^1_{\#_i}(C_i) : = \{ u \in H^1(C_i) \, | \, u \text{ is $1$-periodic in the variable $y_i$}\}$ is the unique  solution to the cell problem
 \begin{equation}
 \label{fluxfinal}
 \left\{ \begin{aligned}
 & \int_{C_i} a_1(y) \big[ \bfnabla N^{(i)}(y) +  \bfe_i  \big] \cdot \overline{ \bfnabla \phi } (y) \, {\rm d}y = 0, \qquad \forall \phi \in H^1_{\#_i}(C_i), \\
 &\int_{C_i } N^{(i)} = 0.
 \end{aligned} \right.
  \end{equation}
Then, for each $\bftheta \in [0,2\pi)^3$, the $\bftheta$-quasi-periodic two-scale limit problem is formulated as follows: For $f \in L^2(\Omega \times Q)$ find $u \in U_\bftheta$ such that
	\begin{equation}
	\label{qphomlimit0}
	\begin{aligned}
	 \sum_{i \in \I^{\bftheta}}&\int_{\Omega} a^{\rm hom}_i \partial_{x_i} u_i (x) \overline{\partial_{x_i} \phi_i}(x) \, {\rm d}x + \int_{\Omega} 	\int_{Q_0} a_0(y) \bfnabla_y u(x,y) \cdot \overline{\bfnabla_y \phi(x,y)} \, {\rm d}y {\rm d}x  \\	& + 	\int_{\Omega} 	\int_{Q} u(x,y)\overline{ \phi(x,y)} \, {\rm d}y {\rm d}x  = 	\int_{\Omega} 	\int_{Q} f(x,y)   \overline{ \phi(x,y) }\, {\rm d}y {\rm d}x, \qquad\quad \forall  \phi \in U_\bftheta.
	\end{aligned}
	\end{equation}
As $a^{\rm hom}_i$ are positive numbers and $a_0^{-1} \in L^\infty(Q_0)$ it follows that the left-hand side of the above problem defines an equivalent inner product on the space $U_\bftheta$, and consequently the existence and uniqueness of solutions $u$ to \eqref{qphomlimit0} are ensured by the Riesz representation theorem.

Setting $\phi = 0$ on $Q_1$ in \eqref{qphomlimit0} gives the equation 
$$
- {\rm div}_y \big( a_0(y) \bfnabla_y u(x,y) \big) + u(x,y) = f(x,y), \hspace{.9cm}x \in \Omega, \,y \in Q_0,
$$
and a subsequent integration by parts in \eqref{qphomlimit0} leads to the variational formula
$$
\begin{aligned}
\sum_{i \in \I^{\bftheta}}\int_{\Omega} a^{\rm hom}_i \partial_{x_i} u_i &(x) \overline{\partial_{x_i} \phi_i}(x) \, {\rm d}x + \sum_{i \in \I^\bftheta} \int_{\Omega} \left(  	\int_{\Gamma_i} a_0(y) \bfnabla_y u(x,y) \cdot \bfn(y)  \, {\rm d}y \right)\overline{\phi_i(x)} {\rm d}x  \\	& + \sum_{i \in \I^{\bftheta}}	|C_i| \int_{\Omega} 	 u_i(x)\overline{ \phi_i (x)} {\rm d}x  = \sum_{i \in \I^{\bftheta}}	 \int_{\Omega} \left( \int_{C_i} f(x,y) \, {\rm d}y \right)  \overline{ \phi_i(x) } {\rm d}x,
\end{aligned}
$$
for all $\bfphi \in L^2(\Omega,\CC^\bftheta)$, such that $\partial_i \phi_i \in L^2(\Omega)$. For each fixed $j \in \I^\bftheta$ we set $\phi_j = \phi$, $\phi \in C^\infty_0(\Omega)$ and $\phi_i = 0$ for $i \neq j$, above. This leads to the $\bftheta$-quasi-periodic two-scale homogenised system of equations.
\begin{equation}
\label{representation}
\begin{aligned}
&\text{for each $j \in \mathcal{I}^\bftheta$:} \quad  \left\{
\begin{aligned}
&\begin{array}{lcr}
- a^{\rm hom}_j \partial_{x_j}^2 u_j(x) + \T_j(u) + |C_j | u_j = \langle f \rangle_{j}(x), & &  \quad x \in \Omega, \\[5pt]
- {\rm div}_y \big( a_0(y) \bfnabla_y u(x,y) \big)  + u(x,y)= f(x,y), & \hspace{.5cm} & x \in \Omega, \,y \in Q_0, 
\end{array} \\
& \hspace{10pt}u_0 = u_j \quad \text{on $\Omega \times  \Gamma_j$}, \hspace{1cm} u_j \nu_j = 0 \quad \text{on $\partial \Omega$}, 
\end{aligned}\right.  \\
&\text{and} \hspace{10pt}
u_0 = 0 \quad \text{on $\Omega \times  \Gamma_i$, for $i \in \I \backslash \I^\bftheta$.}
\end{aligned}
\end{equation}
Here
$$
\begin{aligned}
\T_j(u)(x) = \int_{\Gamma_j} a_0(y) \bfnabla_y u(x,y) \cdot \bfn(y) \, {\rm d}S(y), \qquad \langle f \rangle_j(x) = \int_{C_j} f(x,y) \, {\rm d}y,
\end{aligned}
$$ for $\bfn$ the outer unit normal of $\Gamma_j = \partial C_j \backslash \partial Q$.
We now state the main result of the article.
\begin{theorem}
	\label{blochhomogenisation}
Consider $f_\ep \in L^2(\Omega)$, $f \in L^2(\Omega \times Q)$ such that $f_\ep \qp f$, and $u_\ep$ the solution to \eqref{resolventproblem}. Then $u_\ep$ converges, up to some subsequence, in the ${\bftheta}$-quasi-periodic sense to  $u \in U_\bftheta$ the unique solution to  \eqref{qphomlimit0}, equivalently \eqref{representation}.
\end{theorem}

An immediate consequence of Theorem \ref{blochhomogenisation} is that for each $\bftheta \in [0,2\pi)^3$, the operator $A_\ep$ strong resolvent $\bftheta$-quasi-periodically two-scale converges to the operator $A^{\rm hom}_\bftheta$ associated to problem \eqref{representation}, see Definition \ref{stopcon} in Section \ref{sec:qpcon}. Consequently, Proposition \ref{tsrescon} informs us that the lower semi-continuity of the spectra in the Hausdorff sense is ensured:
$$
\text{ for every $\lambda \in \bigcup_{\bftheta \in [0,2\pi)^3} \sigma(A^{\rm hom}_\bftheta)$ \quad $\exists\,  \lambda_\ep \in \sigma(A_\ep)$ such that $ \lim_{\ep \rightarrow 0 } \lambda_\ep =\lambda$.}
$$
The structure of the limit spectrum $\bigcup\limits_{\bftheta \in [0,2\pi)^3} \sigma(A^{\rm hom}_\bftheta)$ is analysed in Section \ref{sec:limspec} and described in Proposition \ref{prop:limitspecchar}.

\begin{remark}
A seperate issue, not explored here, is the so-called spectral completeness statement, i.e. the question of whether or not the remaining criterion for Hausdorff convergence of spectra is satisfied: does it follow that
$$
\text{ for every $\lambda_\ep$ such that $\lim_{\ep \rightarrow 0}\lambda_\ep =\lambda$, then $\lambda \in \bigcup_{\bftheta \in [0,2\pi)^3} \sigma(A^{\rm hom}_\bftheta)$  } \ ?
$$
In general this will not be true due to the presence of the boundary, and the fact that $Q_0$ intersects the boundary. This leads to the expectation that there exists non-trivial spectrum due to surface waves asymptotically localised near the boundary, cf. \cite{AlCo} for analoguous results in the context of classical locally periodic media. For the case of $\Omega$ being the torus or the whole space the above assertion is expected to hold and will be explored in future works. 
\end{remark}

\section{Proof of the homogenisation theorem}
\label{sec:proof}
\label{proofthm}
This section is dedicated to the proof of Theorem \ref{blochhomogenisation}. To do this, we shall develop an appropriate  quasi-periodic two-scale variation of a powerful method first introduced in \cite{IVKVPS} in the context of standard (periodic) two-scale convergence, i.e.  ${\bftheta}$-quasi-periodic two-scale convergence for ${\bftheta} = 0$.  In what follows $\phi$, $\varphi$ and $\bfPhi$ will denote respectively fixed arbitrary elements of $C^\infty_0(\Omega)$, $C^\infty_{\bftheta}(Q)$ and $C^\infty_{\bftheta}(Q;\CC^3)$.

\subsection{Technical preliminaries}

The following results will be of importance in the proof of the homogenisation theorem.

\begin{lemma}
	\label{extension}
	Let $B$ be the closure of a smooth domain and let $B_1$ be a smooth bounded domain such that $B \subset B_1$ and $A = B_1 \backslash B$ is a connected set. Then every $u \in H^1\big((0,1)\times A\big)$  can be extended to $(0,1)\times B_1$ as a function $\widetilde{u} \in H^1\big((0,1) \times B_1\big)$ such that
	\begin{equation}
	\label{ourext1}
	\begin{aligned}
	\int_{(0,1)\times B_1} | \bfnabla \widetilde{u} |^2 &  \le c  \int_{(0,1)\times A} | \bfnabla u |^2 , \\
	\int_{(0,1)\times B_1} |  \widetilde{u} |^2&  \le c  \int_{(0,1)\times A} | u |^2,
	\end{aligned}
	\end{equation}
	where $c$ does not depend on $u \in H^1((0,1)\times A)$.
\end{lemma}

\begin{proof}
	Suppose $u \in H^1\big((0,1)\times A\big)$. Then, by Fubini's theorem, for almost every $x_1 \in (0,1)$ the function $u(x_1,\cdot)$ belongs to $H^1(A)$ and let $Eu(x_1,\cdot)$ be the Sobolev extension of $u(x_1,\cdot)$ into $B_1$ given in \cite[Lemma 3.2, pg. 88]{JKO}. In particular, one has 
	\begin{equation}
	\label{jkoext1}
	\begin{aligned}
	\int_{B_1} | \bfnabla' Eu (x_1,\cdot)|^2  &  \le c  \int_{A} | \bfnabla' u(x_1,\cdot) |^2, \\
	\int_{B_1} |  Eu(x_1,\cdot) |^2 &  \le c  \int_{A} | u(x_1,\cdot) |^2,
	\end{aligned}
	\end{equation}
	where $c$ does not depend on $u$ nor $x_1$. Here,  $\bfnabla'$ denotes the gradient vector $(0,\partial_{x_2}, \partial_{x_3})$. 
	
	Consider $\widetilde{u}$ given by
	\begin{equation}
	\label{mapping}
	\widetilde{u(x_1,\cdot)} : = Eu(x_1,\cdot), \qquad \text{ a.e. } x_1 \in (0,1).
	\end{equation}
	Then $\bfnabla ' \widetilde{u(x_1 , \cdot)} =\bfnabla ' Eu(x_1, \cdot)  $ and from assertion \eqref{jkoext1} it follows that
	$$
	\begin{aligned}
	\int_{(0,1)\times B_1} | \bfnabla' \widetilde{u} |^2  &\le c \int_{(0,1)\times A} | \bfnabla' u |^2 , \\
	\int_{(0,1)\times B_1} |  \widetilde{u} |^2  &\le c \int_{(0,1)\times A} |  u |^2.
	\end{aligned}
	$$
	To prove \eqref{ourext1}, it remains to demonstrate that $\partial_{x_1} \widetilde{u} \in L^2\big((0,1)\times B_1\big)$ and
	\begin{equation}
	\label{extx1}
	\begin{aligned}
	\int_0^1 \int_{B_1} | \partial_{x_1} \widetilde{u} |^2 (x_1,x') \, {\rm d}x'{\rm d}x_1 & \le c \int_0^1 \int_{A} | \partial_{x_1} u |^2 (x_1,x') \, {\rm d}x'{\rm d}x_1.
	\end{aligned}
	\end{equation}
	For each $t \in \RR$, the difference quotient is given by
	$$
	D_t\widetilde{u}(x_1,x') : = \frac{\widetilde{u}(x_1 + t, x') - \widetilde{u}(x_1,x')}{t},
	$$
	where we have extended $\widetilde{u}$ trivially by zero into $\RR \backslash(0,1)$.
	Notice that  $D_t\widetilde{u}= E D_t  u$, i.e. the extension into $B_1$ of the function $\tfrac{u(x_1 + t, \cdot) - u(x_1,\cdot)}{t} = D_tu(x_1+t,\cdot)$, and consequently
	$$
	\int_0^1 \int_{B_1} | D_t \widetilde{u}  |^2 (x_1,x')\, {\rm d}x'{\rm d}x_1 \le c \int_0^1 \int_{A} |  D_tu |^2 (x_1,x')\, {\rm d}x'{\rm d}x_1. 
	$$
	Since $D_t u$ converges strongly in $L^2\big( (0,1) \times A \big)$ to $\partial_{x_1} u$, it follows that $D_t \widetilde{u}$ is a Cauchy sequence in $L^2\big((0,1)\times B_1\big)$ and this limit can be identified, using the fact that $(D_t \widetilde{u}, \phi )_{L^2} = ( \widetilde{u}, D_{-t} \phi)_{L^2}$ for $\phi \in C^\infty_0\big((0,1)\times B_1\big)$, as $\partial_{x_1} \widetilde{u}$.
	Furthermore, passing to the limit in the above inequality yields \eqref{extx1}.
\end{proof}

\begin{proposition}
	\label{prop:ka}
	Fix ${\bftheta} \in [0,2\pi)^3$. There exists a constant $C=C_{\bftheta} >0$ such that
	$$
	\inf_{v \in V_{\bftheta}}|| u - v ||_{H^1_{\bftheta}(Q)} \le C_{\bftheta}  || \sqrt{a_1} \bfnabla u||_{L^2(Q)}, \qquad \forall u \in H^1_{\bftheta}(Q).
	$$
	Here $V_{\bftheta} = {\rm ker} \sqrt{a_1} \bfnabla_{\bftheta} = \{ u \in H^1_{\bftheta}(Q) \, \big| \, \sqrt{a_1} \bfnabla u \equiv 0 \}$.
\end{proposition}
\begin{proof}
	For each $i \in \I$, let $S_i$ be the cross-section of the cylinder $C_i$. Since $S_i$ are compactly contained in $(0,1)^2$ and have mutually disjoint closures then there exists open $A_i$ such that $\overline{{S_i}} \subset A_i\subset (0,1)^2$ and $A_i $ are mutually disjoint. Let  $\chi_i \in C^\infty_0(A_i)$ be smooth cut-off functions that are identity on $S_i$, we extend $\chi_i$ by zero to $(0,1)^2$.
	
	Now using Lemma \ref{extension}, let $\widetilde{u_i}$ be the extension of $u\vert_{C_i} \in H^1(C_i)$ to $H^1(D_i)$, where $D_i$ is the cylinder whose axis is parallel to $y_i$ with cross-section $A_i$. Note that since $u$ is $\theta_i$-quasi-periodic in the variable $y_i$, then the extension will be also, see \eqref{mapping}. By Lemma \ref{extension} it follows that
	\begin{equation}
	\label{kape2}
	\int_{D_i} | \bfnabla \widetilde{u_i} |^2   \le c  \int_{C_i} | \bfnabla u_i |^2.
	\end{equation}
	For $\theta_i \neq 0$,  the following Poincar\'{e} inequality 
	\begin{equation}
	\label{pitheta}
	\int_{D_i} | \widetilde{u_i} |^2 \le | \theta_i |^{-2} \int_{D_i} | \bfnabla \widetilde{u_i} |^2 
	\end{equation}
	holds\footnote{This follows from noting the lower bound on the spectrum of the laplacian on the space of $H^1\bftheta(D_i)$ functions that are $\theta_i$-quasi-periodic in direction $y_i$.}. For $\theta_i = 0$, one has
	\begin{equation}
	\label{pi0}
	\int_{D_i} | \widetilde{u_i} - \langle\widetilde{u_i}\rangle |^2 \le C \int_{D_i} | \bfnabla \widetilde{u_i} |^2 
	\end{equation}
	for some $C>0$. Here $\langle \widetilde{u_i} \rangle := \tfrac{1}{D_i}\int_{D_i} \widetilde{u_i}$.
	
	Recalling, $\I^{\bftheta} = \{ i \in \I \, | \, \theta_i=0\}$, we set $\widetilde{u} = \sum_{i \in \I^\bftheta}  \chi_i ( \widetilde{u_i} - \langle \widetilde{u_i} \rangle ) + \sum_{i \in \I \backslash \I^\bftheta} \chi_i\widetilde{u_i}$, here $\chi_i$ are taken to be constant in the variable $y_i$ and as above  the complementary directions. It follows that $\widetilde{u} \in H^1_\bftheta(Q)$ and $u - \widetilde{u} \in V_\bftheta$. Note that, by construction and inequalities \eqref{pitheta}, and \eqref{pi0}, one has
	$$
	||  \widetilde{u} ||^2_{H^1(Q)} \le c_\bftheta \sum_{i \in \I} || \bfnabla \widetilde{u_i} ||^2_{L^2(D_i)}.
	$$
	Now, the positivity of $a_1$ on $Q_1$ and \eqref{kape2} imply that the element $v : = u - \widetilde{u}$ of  $V_\bftheta$ is such that 
	$$
	|| u - v ||^2_{H^1(Q)} = || \widetilde{u} ||^2_{H^1(Q)} \le C_\bftheta \int_{Q_1} a_1 | \bfnabla u |^2,
	$$
	and the result follows.

\end{proof}
\subsection{Proof of Theorem \ref{blochhomogenisation}}

Consider the sequence $u_\ep$ of solutions to \eqref{resolventproblem}, i.e.
\begin{equation}
\label{resolventproblem1}
\begin{aligned}
\int_{\Omega} \big(a_1 (\tfrac{x}{\ep}) + \ep^2 a_0 (\tfrac{x}{\ep})\big) \bfnabla u_\ep(x) & \cdot \overline{\bfnabla \phi}(x) \, {\rm d}x \ + \int_{\Omega} u_\ep(x)  \overline{\phi}(x) \, {\rm d}x  = \int_{\Omega} f_\ep(x)  \overline{\phi}(x) \, {\rm d}x, \qquad \forall \phi \in C^\infty_0(\Omega).
\end{aligned}
\end{equation}
for $f_\ep \qp f$, and recall, cf. \eqref{apbounds}, that 
\begin{equation}\label{apbounds1}
\begin{aligned}
\sup_{\ep} \left( || \sqrt{a_1}(\tfrac{\cdot}{\ep}) \bfnabla u_\ep ||_{L^2(\Omega;\CC^3)} + || \ep \bfnabla u_\ep ||_{L^2(\Omega;\CC^3)} + \ \ || u_\ep ||_{L^2(\Omega)} \right) < \infty.
\end{aligned}
\end{equation}
Consequently, Proposition \ref{prop:hc} informs us that a subsequence of $u_\ep$ $\bftheta$-quasi-periodic two-scale converges to some $u \in L^2(\Omega;H^1_\bftheta(Q))$, and moreover $\ep \bfnabla u_\ep \qp \bfnabla_y u$. Let us study the structure of this limit $u$ further. 

We begin by introducing the densely defined unbounded linear operator $\sqrt{a_1} \bfnabla_{\bftheta} :H^1_{\bftheta}(Q) \subset L^2(Q) \rightarrow L^2(Q;\CC^3)$ which is given by the action 
$$
w \mapsto \sqrt{a_1} \bfnabla w , \qquad \text{ for $w \in H^1_{\bftheta}(Q)$}.
$$
We now argue that a generalised Weyl's decomposition holds, which was first introduced and proved for the case ${\bftheta} =0$ in \cite{IVKVPS}.
\begin{lemma}
	\label{lem:Weylsdecom}
 Let  $(\sqrt{a_1} \bfnabla_\bftheta)^*$ denote the adjoint of $\sqrt{a_1} \bfnabla_\bftheta$. Then, the orthogonal decomposition 
\[
	L^2(Q;\CC^3) = {\rm ker}\big(  (\sqrt{a_1} \bfnabla_{\bftheta})^* \big) \oplus {\rm Ran}( \sqrt{a_1} \bfnabla_{\bftheta})\]
holds.
\end{lemma}
\begin{remark}
Lemma \ref{lem:Weylsdecom} is a generalisation of the well-known fact that (periodic) divergence-free vector fields are mutually orthogonal to gradients of (periodic) potentials in $L^2$.  In fact, this classical result can be deduced from the above lemma by (formally)\footnote{In fact, as expected the proof of this statement for $I$ is much easier as $I$ is positive where as $\sqrt{a_1}$ is non-negative.} setting $\sqrt{a_1} = I$ on $\square$.
\end{remark}
\begin{proof}[Proof of Lemma \ref{lem:Weylsdecom}] By the Banach closed ranged theorem, this result will follow if we demonstrate that the range of $ \sqrt{a_1} \bfnabla_{\bftheta} $ is closed, and this fact is implied by Proposition \ref{prop:ka}.
	
	Indeed, suppose $u_n \in {\rm Ran}( \sqrt{a_1} \bfnabla_{\bftheta})$ converges strongly in $L^2(Q;\CC^3)$ to some $u$ as $n \rightarrow \infty$, i.e. there exists $w_n \in H^1_{\bftheta}(Q)$ such that $\sqrt{a_1} \bfnabla w_n$ converges strongly in $L^2(Q;\CC^3)$ to $u$. By Proposition \ref{prop:ka}, the sequence $w_n^\perp$, where $w_n^\perp$ denotes the orthogonal projection of $w_n$ onto  the orthogonal complement $V^\perp_{\bftheta}$ of $V_{\bftheta}$ in $H^1_{\bftheta}(Q)$, is a Cauchy sequence in $H^1_{\bftheta}(Q)$ and therefore converges, up to some subsequence, to $w \in H^1_{\bftheta}(Q)$. In particular, $\sqrt{a_1} \bfnabla w_n = \sqrt{a_1} \bfnabla w^\perp_n$ converges strongly in $L^2(Q)$ to $\sqrt{a_1} \bfnabla w$ and, consequently $u =\sqrt{a_1} \bfnabla w$. Hence, the range of $\sqrt{a_1} \bfnabla_{\bftheta}$ is closed.
\end{proof}
Let us now describe $u$ in detail.
\begin{lemma}
	\label{lem:qpcompactness1}
	The function $u$ belongs to the Bochner space $L^2(\Omega; V_\bftheta)$. 
\end{lemma}
\begin{proof}
	Recall that
	$$
	V_{\bftheta} = \{ v \in H^1_{\bftheta}(Q) \, \big| \, \bfnabla v = 0 \text{ in $Q_1$} \} = {\rm ker} ( \sqrt{a_1} \bfnabla_{\bftheta} ),
	$$
	and so we aim to show that $\sqrt{a_1} \bfnabla_{\bftheta} u = 0$.

	On the one hand we deduce from \eqref{apbounds1} and  \eqref{mvprop}  that
	$$
	\lim_{\ep \rightarrow 0}\ep \int_{\Omega} a_1(\xoe) \bfnabla u_\ep \cdot \overline{\phi(x)\bfPhi(\xoe)} \, {\rm d}x = 0.
	$$
	Yet, on the other hand, Proposition \ref{prop:qptests} and the assertion $\ep \bfnabla u_\ep \qp \bfnabla_y u$ imply 
	$$
	\begin{aligned}
	\lim_{\ep \rightarrow 0} \ep \int_{\Omega} a_1(\xoe) \bfnabla u_\ep \cdot \overline{\phi(x)\bfPhi(\xoe)}\, {\rm d}x & =  \lim_{\ep \rightarrow 0} \int_{\Omega} \ep  \bfnabla u_\ep \cdot \overline{\phi(x)a_1(\xoe)\bfPhi(\xoe)}\, {\rm d}x \\
	& = \int_{\Omega}\int_Q a_1(y) \bfnabla_y u(x,y)\cdot \overline{\phi(x) \bfPhi(y)} \, {\rm d}x.
	\end{aligned}
	$$
	Therefore, as finite sums of $\phi(x)\bfPhi(y)$ are dense in $L^2(\Omega\times Q;\CC^3)$ it follows that $a_1 \bfnabla_y u = 0$ and since $\sqrt{a_1}^{-1} \in L^\infty(Q)$  we find that $u \in L^2(\Omega ; V_{\bftheta})$.
\end{proof}
The following result is of fundamental importance in characterising the ($\bftheta$-quasi-periodic)  limit of the flux $a_1(\tfrac{\cdot}{\ep}) \bfnabla u_\ep$ in terms of the limit $u$ of the function $u_\ep$. Put another way, this identity is crucial for determining the homogenised coefficients.
\begin{lemma}
	\label{lem:qpcompactness2}
	There exists $\bfxi \in L^2(\Omega \times Q ; \CC^3)$ such that, up to a subsequence, $\sqrt{a_1}(\tfrac{\cdot}{\ep}) \bfnabla u_\ep \qp \bfxi$. Moreover, $\bfxi$ belongs to the Bochner space $ L^2 \big(\Omega ; {\rm ker} \big( (\sqrt{a_1} \bfnabla_{\bftheta})^* \big)\big)$ and the pair $(u,\bfxi)$ satisfies the identity
	\begin{equation}
	\label{qpfluxidentity}
	\begin{aligned}
	\int_{\Omega} \int_Q \bfxi(x,y) \cdot \overline{ \phi(x) \bfPsi(y) } \, {\rm d}y{\rm d}x = - & \int_{\Omega} \int_Q \sqrt{a_1}(y) u(x,y) \overline{ \bfnabla_x \phi(x) \cdot \bfPsi(y) } \, {\rm d}y{\rm d}x, \\
	& \forall \phi \in C^\infty(\Omega), \bfPsi \in  {\rm ker} \big((\sqrt{a_1} \bfnabla_{\bftheta})^* \big).
	\end{aligned}
	\end{equation}
\end{lemma}
\begin{proof}
	By Proposition \ref{prop:qpcompact} and \eqref{apbounds1} there exists $\bfxi \in L^2(\Omega \times Q;\CC^3)$ such that, up to a subsequence that we discard, one has
	\begin{align}
	\sqrt{a_1}(\tfrac{\cdot}{\ep})\bfnabla u_\ep & \qp \bfxi.\label{qplimits3}
	\end{align}{ \ }

	To prove $\bfxi \in L^2 \big(\Omega ; {\rm ker} \big( (\sqrt{a_1} \bfnabla_{\bftheta})^* \big)$, we take  in \eqref{resolventproblem1} test functions of the form $ \ep \phi(x) \varphi(\xoe)$, $\phi \in C^\infty_0(\Omega)$, $\varphi \in C^\infty_{\bftheta}(Q)$, and use \eqref{apbounds1}, \eqref{qplimits3} to pass to the limit in $\ep$ and deduce that
	$$
	\int_{\Omega} \int_Q \sqrt{a_1}(y) \bfxi(x,y) \cdot \overline{\phi(x) \bfnabla_y\varphi(y)} \, {\rm d}y{\rm d}x = 0.
	$$
	Therefore, for almost every $x \in \Omega$ one has
	$$
	\int_Q \sqrt{a_1}(y) \bfxi(x,y) \cdot \overline{ \bfnabla_y\varphi(y)} \, {\rm d}y = 0, \qquad \forall \varphi \in H^1_{\bftheta}(Q),
	$$
	and, hence by Lemma \ref{lem:Weylsdecom} it follows that $\bfxi(x,y) \in  L^2 \big(\Omega ; {\rm ker} \big( (\sqrt{a_1} \bfnabla_{\bftheta})^* \big)$.

	Let us now prove assertion \eqref{qpfluxidentity}. Henceforth, we consider $\Psi \in {\rm ker}\big( (\sqrt{a_1} \bfnabla_{\bftheta})^* \big)$ to be ${\bftheta}$-quasi-periodically extended to $\RR^3$. We shall prove below the following ``integration by parts" formula:
	\begin{equation}
	\label{qpIBP}
	\begin{aligned}
	\hspace{-5pt}\int_{\Omega} \sqrt{a_1} (\xoe) \bfnabla u_\ep (x) \cdot \overline{\phi(x) \bfPsi(\xoe)} \, {\rm d}x = & - \int_{\Omega} \sqrt{a_1} (\xoe)  u_\ep (x) \overline{\bfnabla_x \phi(x) \cdot \bfPsi(\xoe)} \, {\rm d}x, \\
	& \forall \phi \in C^\infty(\Omega), \bfPsi \in  {\rm ker} \big((\sqrt{a_1} \bfnabla_{\bftheta})^* \big).
	\end{aligned}
	\end{equation}
	Using Proposition \ref{prop:qptests}, \eqref{qplimits3} and the convergence $u_\ep \qp u$, we pass to the limit in the above formula to readily arrive at \eqref{qpfluxidentity}.
	
	To prove \eqref{qpIBP}, it is sufficient to prove the following:  for every $w \in H^1(\RR^3)$ one has  
	\begin{equation}
	\label{qpfluxe1}
	\int_{\RR^3} \sqrt{a_1} (\xoe) \bfnabla w(x) \cdot \overline{ \Psi(\xoe)} \, {\rm d}x = 0.
	\end{equation}
	
	Indeed, \eqref{qpIBP} follows from utilising \eqref{qpfluxe1} and the following facts: for $\phi \in C^\infty(\Omega)$ then $u_\ep \phi$ belongs to $H^1_0(\Omega)$, as $u_\ep \in H^1_0(\Omega)$,  and can be trivially extended to $H^1(\RR^3)$, and that
	\begin{multline*}
	\int_{\Omega} \sqrt{a_1} (\xoe) \bfnabla u_\ep (x) \cdot \overline{\phi(x) \Psi(\xoe)} \, {\rm d}x 
	= \int_{\Omega} \sqrt{a_1} (\xoe) \bfnabla (u_\ep \overline{\phi})(x)  \cdot \overline{\Psi(\xoe)} \, {\rm d}x -  \int_{\Omega} \sqrt{a_1} (\xoe) u_\ep(x) \overline{\bfnabla_x \phi(x) \cdot  \Psi(\xoe)} \, {\rm d}x.
	\end{multline*}
	
	Let us now prove \eqref{qpfluxe1}. For $Q^{(z)}_\ep = \prod_{i=1}^{3}\ep(z_i, z_i+1)$, $z \in \ZZ^3$, it follows that
	$$
	\begin{aligned}
	\int_{\RR^3} \sqrt{a_1} (\xoe) \bfnabla w(x) \cdot \overline{ \Psi(\xoe)} \, {\rm d}x & = \sum_{z \in \ZZ^3} \int_{Q^{(z)}_\ep} \sqrt{a_1} (\xoe) \bfnabla w(x) \cdot \overline{ \Psi(\xoe)} \, {\rm d}x \\ 
	&= \ep^3 \sum_{z \in \ZZ^3} \int_{Q} \sqrt{a_1} (y) \bfnabla w(\ep y+ \ep z) \cdot \overline{ \exp(\i {\bftheta} \cdot z )\Psi(y)} \, {\rm d}y,
	\end{aligned}
	$$
	where the last equality comes from the change of variables $x=\ep(y+z)$ and recalling that $a_1(y)$ is periodic and $\Psi$ is ${\bftheta}$-quasi-periodic. By noting, for $w \in H^1(\RR^3)$, that 
	$$
	w_\ep(y) : = \sum_{z \in \ZZ^3}w(\ep y+ \ep z) \exp(-\i {\bftheta} \cdot z), \qquad y \in Q,
	$$
	is an element of $H^1_{\bftheta}(Q)$, and that
	$$
	\bfnabla w_\ep(y) : = \ep \sum_{z \in \ZZ^3} \bfnabla w(\ep y+\ep z) \exp(-\i {\bftheta} \cdot z), \qquad y \in Q,
	$$
	the identity \eqref{qpfluxe1} follows.
\end{proof}
We are now ready to describe the properties of the macroscopic part of $u$ and express the flux $\bfxi$ in terms of $u$. 

\begin{lemma}
	\label{lem:qpcompactness3}
	Let $(u,\bfxi)$, $u \in L^2(\Omega; V_\bftheta)$ and $\bfxi \in L^2(\Omega; {\rm ker}\big( (\sqrt{a_1}\bfnabla_\bftheta)^*\big)$, be a pair which satisfies the identity \eqref{qpfluxidentity}. Then, $u \in U_\bftheta$, see \eqref{limitspace}. That is, for every for $i \in \I^{\bftheta} = \{ i \in \I \, | \, \theta_i = 0 \}$, one has $u = u_i$ on $\Omega \times C_i$, where  $\partial_i u_i \in L^2(\Omega)$ with $u_i \nu_i = 0$ on $\partial \Omega$, for $\bfnu$ the outer unit normal to $\partial \Omega$. Furthermore, 
	
	\begin{equation}
	\label{fluxrep}
	\hspace{-5pt}	\bfxi(x,y)=   \sqrt{a_1}(y)\sum_{i \in \I^{\bftheta}} \partial_{x_i}u_{i}(x) \mathds{1}_{C_i}(y)  [     \bfnabla_y N^{(i)}(y) + \bfe_i ], \qquad \text{ $x \in \Omega, y \in Q_1$.}
	\end{equation}
	
	Here, $N^{(i)}$ solve \eqref{fluxfinal}.
\end{lemma}
The following result immediately follows from the above lemma.
\begin{proposition}
	\label{corhommatrix}
	For every for $i \in \I^{\bftheta}$, one has  
	$$
	\int_{C_i} \sqrt{a}_1(y)\bfxi(x,y) \, {\rm d}y =  a^{\rm hom}_i \partial_{x_i}u_{i}(x) \bfe_i, \qquad   \text{ for almost every $x \in \Omega$.}
	$$
	
	Here, $a^{\rm hom}$ are given by \eqref{ahom}, i.e. 
	$$
	a^{\rm hom}_i =\int_{C_i} a_1(y)  [\partial_{y_i} N^{(i)}(y) + 1] \, {\rm d}y > 0,
	$$
	for $N^{(i)} \in H^1_{\#_i}(C_i) = \{ u \in H^1(C_i) \, | \, u \text{ is $1$-periodic in the variable $y_i$}\}$ is the unique  solution to the cell problem
	\begin{equation}
	\label{1fluxfinal}
	\left\{ \begin{aligned}
	& \int_{C_i} a_1(y) \big[ \bfnabla N^{(i)}(y) +  \bfe_i  \big] \cdot \overline{ \bfnabla \phi } (y) \, {\rm d}y = 0, \qquad \forall \phi \in H^1_{\#_i}(C_i), \\
	&\int_{C_i } N^{(i)} = 0.
	\end{aligned} \right.
	\end{equation}
\end{proposition}
\begin{proof}
	Equation  \eqref{fluxrep} implies 
	$$
	\int_{C_i} \sqrt{a}_1(y)\bfxi(x,y) \, {\rm d}y = \partial_{x_i}u_i(x) \int_{C_i} a_1(y) [ \bfnabla_y N^{(i)}(y) + \bfe_i ] \, {\rm d}y.
	$$
	For each $j \in \{ 1,2,3 \} \backslash \{ i \}$, we set $\phi = y_j$ in \eqref{1fluxfinal} to determine that
	$$
	\int_{C_i} a_1(y) \big[ \bfnabla_y N^{(i)}(y) +  \bfe_i  \big] \cdot \overline{ \bfnabla  y_j} (y) \, {\rm d}y  = \int_{C_i} a_1(y) \big[ \bfnabla_y N^{(i)}(y) +  \bfe_i  \big]  \, {\rm d}y \cdot \bfe_j  = 0.  
	$$
	Hence, it follows that
	$$
	\begin{aligned}
	\int_{C_i} \sqrt{a}_1(y) \bfxi(x,y) \, {\rm d}y &= \partial_{x_i}u_{i}(x) \int_{C_i} a_1(y) [\bfnabla_y N^{(i)}(y) \cdot \bfe_i + 1 ] \bfe_i \, {\rm d}y \\
	& = a^{\rm hom}_i \partial_{x_i}u_{i}(x) \bfe_i,
	\end{aligned}
	$$
	for	almost every $x \in \Omega$. Finally, from \eqref{1fluxfinal}  it follows that
	\begin{flalign*}
	\int_{C_i} a_1(y) [\bfnabla_y N^{(i)}(y) \cdot \bfe_i + 1 ] &  = \int_{C_i} a_1(y) [ \bfnabla_y N^{(i)}(y) + \bfe_i  ] \cdot \bfe_i \\
	& = \int_{C_i} a_1(y) [ \bfnabla_y N^{(i)}(y) + \bfe_i] \cdot \overline{[ \bfnabla_y N^{(i)}(y) + \bfe_i  ] }.
	\end{flalign*}
	Then, the positivity of $a^{\rm hom}_i$ can be seen by the inequality 
	$$
	\int_{C_i} a_1(y) [ \bfnabla_y N^{(i)}(y) + \bfe_i] \cdot \overline{[ \bfnabla_y N^{(i)}(y) + \bfe_i  ] } \ge || a^{-1}_1 ||_{L^\infty(Q_1)}^{-1} \int_{C_i} | \bfnabla_y \big( N^{(i)}(y) + y_i \big) |^2 \, {\rm d}y, 
	$$
	and noting that the right-hand side of this inequality can not be zero for this would contradict the periodicity of $N^{(i)}$ in the $y_i$ variable.
\end{proof}
\begin{proof}[Proof of Lemma \ref{lem:qpcompactness3}]
	As $u \in L^2(\Omega; V_{\bftheta})$, see Lemma \ref{lem:qpcompactness1}, then $u$ is constant in each fibre $C_i$, $i \in \I$. Now if ${\theta}_i \neq 0$ then $u$ is necessarily zero in $C_i$. On the other hand, if ${\theta}_i = 0$, i.e. $i \in \I^{\bftheta}$, then $u(x,y) = u_i(x)$ for $x \in \Omega$, $y \in C_i$. That is, $u = u_i$ on $\Omega \times C_i$ for some $\bfu \in L^2(\Omega, \CC^\bftheta)$, where we recall that $\CC^\bftheta$ is the closed subspace of $\CC^{3}$ spanned by $\{ \bfe_i\}_{i \in \I^\bftheta}$.

	Let us now demonstrate that $u$ belongs to the Hilbert space $U_\bftheta.$  By substituting $u=u_i$ on $\Omega \times C_i$, $i \in \I$, into \eqref{qpfluxidentity}, we deduce that
	\begin{equation}
	\label{qpflux2}
	\begin{aligned}
	\int_{\Omega} \int_Q \bfxi(x,y) \cdot \overline{ \phi(x) \Psi(y) } \, {\rm d}y{\rm d}x = - & \sum_{i \in \I^{\bftheta}} \int_{\Omega} \int_{C_i} \sqrt{a_1}(y) u_i(x) \overline{ \bfnabla_x \phi(x) \cdot \Psi(y) } \, {\rm d}y{\rm d}x, \\
	& \forall \phi \in C^\infty(\Omega), \Psi \in  {\rm ker} \big((\sqrt{a_1} \bfnabla_{\bftheta})^* \big).
	\end{aligned}
	\end{equation}
	For fixed $j \in \I^{\bftheta}$, we will show directly below that there exists a function $\bfPsi^{(j)}{\ker} \big((\sqrt{a_1} \bfnabla_\bftheta)^* \big)$ such that
	\begin{equation}
	\label{nondegeneracy}
	\begin{aligned}
	\int_{C_i} \sqrt{a_1} \overline{\bfPsi^{(j)}} = {\bf 0} \quad i \neq j, & \quad  \text{ and } \quad & \int_{C_j} \sqrt{a_1} \overline{ \bfPsi^{(j)}} 
	= \bfe_j.
	\end{aligned}
	\end{equation}
	Therefore 
	$$
	\sum_{i \in \I^\bftheta}u_i(x) \bfnabla_x \phi \cdot  \int_{C_i} \sqrt{a_1}(y) \overline{ \bfPsi^{(j)}(y) } \, {\rm d}y = u_j(x) \overline{\partial_{x_j}\phi(x)}, \qquad \text{a.e. } x \in \Omega,
	$$
	and consequently substituting $\bfPsi^{(j)}$ into \eqref{qpflux2} gives
	\begin{equation*}
	\begin{aligned}
	\int_{\Omega} \left( \int_Q \bfxi(x,y) \cdot  \overline{\bfPsi^{(j)}} \, {\rm d}y \right)\overline{ \phi(x)}{\rm d}x = -  \int_{\Omega}  u_j(x) \overline{ \partial_{x_j} \phi(x) } \, {\rm d}x, &  \\
	\forall \phi \in & C^\infty(\Omega).
	\end{aligned}
	\end{equation*}
	That is, $\partial_{x_j} u_j(x) = \int_Q \bfxi(x,y) \cdot  \overline{\bfPsi^{(j)}} \, {\rm d}y \in L^2(\Omega)$ and $u_j \nu_j = 0$ on $\partial \Omega$ where $\bfnu$ is the outer unit normal to $\partial \Omega$, i.e. $u \in U_\bftheta$ if \eqref{nondegeneracy} holds.
	
	To show \eqref{nondegeneracy}, we note that under the geometric assumptions on cylinders $C_i$, $i \in \I$, there exists a function 
	\begin{equation}
	\label{specialchi}
	\hspace{-10pt}\begin{aligned}
	\chi_i \in C^\infty(Q) \text{ such that $\chi_i = 1$ on $C_i$,  ${\rm supp}(\chi_i)$ compactly contained in $Q$}  \\
	\text{and ${\rm supp}(\chi_i) \cap \overline{C_k} = \emptyset$ for $k\neq i$}.
	\end{aligned}
	\end{equation}
	Then, for each $j \in \I^\bftheta = \{ i \in \I \, | \, \theta_i =0 \}$,  the function $\bfPsi^{(j)} = \tfrac{1}{|C_j| \sqrt{a_1}} \chi_j \bfe_j$ clearly satisfies \eqref{nondegeneracy}. Furthermore, $\bfPsi^{(j)}$ belongs to ${\ker} \big((\sqrt{a_1} \bfnabla)^* \big)$: Indeed, as  ${\theta}_j =0 $, an element $\phi \in H^1_{\bftheta}(Q)$ is $1$-periodic in the variable $y_j$, and it follows
	$$
	\int_{Q} \tfrac{1}{\sqrt{a_1}} \chi_j\bfe_j \cdot \sqrt{a_1} \bfnabla_y \phi  = \int_{C_j} \partial_{y_j} \phi = 0.
	$$
	Therefore, \eqref{nondegeneracy} holds.

	{ Let us now demonstrate \eqref{fluxrep}}.
	For $i \in \I^{\bftheta}$, and almost every $x \in \Omega$, notice that
	$$
	\begin{aligned}
	\sqrt{a_1}(y){u_i}(x)\sum_{\substack{j \in \{1,2,3\} \\ j \neq i}} & \bfe_j \partial_{x_j} \phi (x) \\
	& = \sqrt{a_1}(y) \bfnabla_y \left( u_i(x) \sum_{\substack{j \in \{1,2,3\} \\ j \neq i}} \partial_{x_j} \phi(x) y_j \right), \quad \text{a.e.}\ y \in C_i,
	\end{aligned}
	$$
	and, by the geometric assumption of the cylinders, we can extend $y_j$ into $Q$ such that the extensions are elements of $H^1_{\bftheta}(Q)$ and equal to zero on $C_j$. Therefore, it follows that
	$$
	\begin{aligned}
	\int_{\Omega}\int_{C_i} &\sqrt{a_1}(y){u_i}(x)  \sum_{\substack{j \in \{1,2,3\} \\ j \neq i}} \partial_{x_j} \phi (x)  \Psi_j(y) \, {\rm dy} {\rm dx}  \\ &= \int_{\Omega}\int_{Q}  \sqrt{a_1}(y) \bfnabla_y \left( u_i(x) \sum_{\substack{j \in \{1,2,3\} \\ j \neq i}} \partial_{x_j} \phi(x) y_j \right) \cdot  \bfPsi(y) \, {\rm dy} {\rm dx} = 0.
	\end{aligned}
	$$
	Consequently, \eqref{qpflux2} takes the form
	$$
	\begin{aligned}
	\int_{\Omega} \int_Q \bfxi(x,y) \cdot \overline{ \phi(x) \bfPsi(y) } \, {\rm d}y{\rm d}x = - & \sum_{i \in \I^{\bftheta}} \int_{\Omega} \int_{C_i} \sqrt{a_1}(y) u_i(x) \overline{ \partial_{x_i} \phi(x)  \Psi_i(y) } \, {\rm d}y{\rm d}x, \\
	& \forall \phi \in C^\infty(\Omega), \bfPsi \in  {\rm ker} \big((\sqrt{a_1} \bfnabla_{\bftheta})^* \big).
	\end{aligned}
	$$
	Integrating by parts above, which is permissible since $\partial_{i} u_i \in L^2(\Omega)$, we deduce that 
	$$
	\begin{aligned}
	\int_{\Omega} \int_Q \bfxi(x,y) \cdot \overline{ \phi(x) \bfPsi(y) } \, {\rm d}y{\rm d}x =  & \sum_{i \in \I^{\bftheta}} \int_{\Omega} \int_{C_i} \sqrt{a_1}(y) \partial_{x_i} u_i(x) \overline{\phi(x)  \Psi_i(y) } \, {\rm d}y{\rm d}x, \\
	& \forall \phi \in C^\infty_0(\Omega), \bfPsi \in  {\rm ker} \big((\sqrt{a_1} \bfnabla_{\bftheta})^* \big).
	\end{aligned}
	$$
	That is, for almost every $x$, $\bfxi(x,\cdot)$ is the projection onto $\ker\big((\sqrt{a_1} \bfnabla_{\bftheta})^* \big)$ of the function 
	$$
	\bfw(x,\cdot) = \sqrt{a_1}(\cdot)\sum_{i \in \I^{\bftheta}} \partial_{x_i}u_{i}(x) \mathds{1}_{C_i}(\cdot) \bfe_i.
	$$
	For each $i \in \I^{\bftheta}$, let $\chi_i$ given by \eqref{specialchi}, and we introduce $\widetilde{N^{i}} \in H^1(Q)$ the extension into $Q$, given by Lemma \ref{extension}, of the function $N^{(i)} \in H^1_{\#_i}(C_i)$ that solves \eqref{1fluxfinal}. It follows that $\sum_{i \in \I^{\bftheta}}  \chi_i \widetilde{N^{(i)}}$ belongs to $H^1_{\bftheta}(Q)$ and
	$$
	\int_Q a_1 \chi_i  [\bfnabla \widetilde{N^{(i)}} + \bfe_i] \cdot \overline{\bfnabla \phi} = \int_{C_i} a_1 [\bfnabla N^{(i)} + \bfe_i ] \cdot \overline{ \bfnabla \phi } = 0
	$$
	for all $\phi \in H^1_{\bftheta}(Q)$. That is, $\sqrt{a_1}\chi_i  [  \bfnabla \widetilde{N^{(i)}} + \bfe_i]$ belongs to  $\ker\big((\sqrt{a_1} \bfnabla_{\bftheta})^* \big)$. Obviously
	\begin{equation*}
	\begin{aligned}
	\bfw(x,y) = \bfw(x,y) + \sqrt{a_1}(y) \sum_{i \in \I^{\bftheta}} \partial_{x_i} u_i (x)  \chi_i(y) \bfnabla_y \widetilde{N^{(i)}}(y)  -  \sqrt{a_1}(y) \sum_{i \in \I^{\bftheta}} \partial_{x_i} u_i (x)\chi_i(y) \bfnabla_y \widetilde{N^{(i)}}(y),
	\end{aligned}
	\end{equation*}
	and
	$$
	\sqrt{a_1}\chi_i  \bfnabla \widetilde{N^{(i)}}  = \sqrt{a_1}\bfnabla  ( \chi_i \widetilde{N^{(i)}}),
	$$
	since $\chi_i$ is piece-wise constant on $C$. Consequently, as $\bfxi(x,\cdot)$ is the projection of $w(x,\cdot)$ onto $\ker\big( (\sqrt{a_1} \bfnabla)^* \big)$, we have
	$$
	\begin{aligned}
	\bfxi(x,y) &= \bfw(x,y) + \sqrt{a_1}(y) \sum_{i \in \I^{\bftheta}} \partial_{x_i} u_i (x)\chi_i(y) \bfnabla_y \widetilde{N^{(i)}}(y)   =   \sqrt{a_1}(y)\sum_{i \in \I^{\bftheta}} \partial_{x_i}u_{i}(x)   [  \mathds{1}_{C_i}(y) \bfe_i +  \chi_i(y)  \bfnabla_y \widetilde{N^{(i)}}(y) ],
	\end{aligned}
	$$
	Hence, \eqref{fluxrep} holds and the proof is complete.
	
\end{proof}

We now conclude with the proof of Theorem \ref{blochhomogenisation}. That is, we show that $u$ solves \eqref{qphomlimit0}. We being by stating that under the assumption that $\Omega$ is star-shaped, standard pull-back and mollification type arguments prove that functions smooth in $x$ are dense in the Hilbert space $U_\bftheta$. Therefore, it is sufficient to show \eqref{qphomlimit0} holds for such test functions $\phi$. Let us take such a $\phi$ and consider the test functions $\phi_\ep(x) = \phi(x,\tfrac{x}{\ep})$, $x \in \Omega$ in \eqref{resolventproblem1}. Utilising the convergences
$$
u_\ep \qp u, \qquad \ep \bfnabla u_\ep \qp \bfnabla_y u, \qquad \sqrt{a_1}(\tfrac{\cdot}{\ep}) \bfnabla u_\ep \qp \bfxi,
$$
we  pass to the limit $\ep \rightarrow 0$ in \eqref{resolventproblem1} to deduce that 
$$
\begin{aligned}
\int_\Omega \int_{C} \sqrt{a_1}(y)\bfxi (x,y) \cdot &\overline{\bfnabla_x \phi}(x,y) \, {\rm d}y{\rm d}x \, +  \int_\Omega \int_{Q_0} a_0(y) \bfnabla_y u(x,y) \cdot \overline{\bfnabla_y \phi}(x,y) \, {\rm d}y{\rm d}x  \\
& + \int_\Omega \int_{Q} u(x,y) \overline{\phi}(x,y) \, {\rm d}y{\rm d}x  = \int_{\Omega}\int_Q  f(x,y) \overline{\phi(x,y)} \, {\rm d}y{\rm d}x.
\end{aligned}
$$
Then, as $\phi = \phi_i$ on $\Omega \times C_i$, with $\phi \neq 0$ only if $i \in \I^\bftheta$,  Proposition \ref{corhommatrix} implies that
$$
\begin{aligned}
\int_\Omega \int_{C} \sqrt{a_1}(y)\bfxi (x,y) \cdot \overline{\bfnabla_x \phi}(x,y) \, {\rm d}y{\rm d}x & = \sum_{i \in \I^\bftheta}  \int_\Omega  \left( \int_{C_i} \sqrt{a_1}(y)\bfxi (x,y)  \, {\rm d}y \right) \cdot \overline{\bfnabla_x \phi_i}(x){\rm d}x \\
& =\sum_{i \in \I^\bftheta}  \int_\Omega  a^{\rm hom}_i \partial_{x_i} u_i(x) \overline{\partial_{x_i} \phi_i}(x){\rm d}x
\end{aligned}
$$
and \eqref{qphomlimit0} follows.

\section{Quasi-periodic two-scale limit operator}
\label{sec:qplimop}
For $\bftheta \in [0,2\pi)^3$, we consider the subspace $H$ which is the closure of $U_\bftheta$ in $L^2(\Omega \times Q)$, i.e.
$$
 \begin{aligned}
H = \big\{ u \in L^2(\Omega \times Q) \, \big| \, u = u_i &\text{ on $\Omega \times C_i$} \\ 
&\text{ for some $\bfu \in L^2(\Omega;\CC^\bftheta)$ } \big\}.
\end{aligned}
$$ 
Indeed, for $f \in H$, we have $f = f_i$ on $\Omega \times C_i$, $i \in \I$, and consequently we deduce that
$$
|| f ||_{L^2(\Omega;L^2(Q_1))}^2 = \sum_{i \in \I^\bftheta} |C_i| || f_i ||_{L^2(\Omega)}^2  ,
$$
and therefore, $H$ is closed in $L^2(\RR^3; L^2( Q))$. It is also straightforward to show that $U_\bftheta$ is dense in $H$. 
Defining on $U_\bftheta$ the form  
$$
Q_\bftheta(u,v) : = \sum_{i \in \I^\bftheta} \int_{\Omega} a^{\rm hom}_i \partial_{x_i} u_i (x) \overline{ \partial_{x_i} v_i (x) } \, {\rm d}x  + \,  \int_{\Omega} \int_{Q_0} a_0(y) \bfnabla_y u(x,y) \cdot \overline{ \bfnabla_y v(x,y)} \, {\rm d}y{\rm d}x,
$$
we find that, since $a^{\rm hom}_i$ are positive constants and $a^{-1}_0 \in L^\infty(Q_0)$, $Q_\bftheta$ is closed when considered as a form on $H$. Setting $A^{\rm hom}_\bftheta : D(A^{\rm hom}_\bftheta) \subset  H \rightarrow H$ to be the unbounded self-adjoint operator generated by $Q_\bftheta$, for $f \in L^2(\RR^3; Q)$ the $\bftheta$-quasi-periodic two-scale homogenised limit problem \eqref{qphomlimit0} takes the form $A^{\rm hom}_\bftheta u = P_\bftheta f$. Here, $P_\bftheta : L^2(\RR^3; L^2(Q)) \rightarrow H$ is the orthogonal projection given by
$$
P_\bftheta f(x,y) = \left\{ \begin{array}{lcr}
\int_{C_i} f(x,y)  \, {\rm d}y, & & y \in C_i, \\[2pt]
f(x,y), & & y \in Q_0.
\end{array} \right. 
$$

An immediate consequence of Theorem \ref{blochhomogenisation} is that for each $\bftheta \in [0,2\pi)^3$, the operator $A_\ep$ strong $\bftheta$-quasi-periodic two-scale resolvent converges to $A^{\rm hom}_\bftheta$, see Section \ref{sec:qpcon} definition \ref{stopcon}. 

\subsection{Spatial operators}
\label{sec:spatialspec}
Introducing the notation
$$
D : = \left( \begin{matrix}
\partial_{x_1} & 0 & 0 \\ 0& \partial_{x_2}  & 0 \\ 0 & 0 &\partial_{x_3}
\end{matrix} \right), \qquad A^{\rm hom} := \left( \begin{matrix}
a^{\rm hom}_1 & 0 & 0 \\ 0&  a^{\rm hom}_2  & 0 \\ 0 & 0 & a^{\rm hom}_3
\end{matrix} \right),
$$
we consider the  Hilbert space
$$
H_\bftheta : = \left\{ \bfu \in L^2(\Omega; \CC^\bftheta)\, |\, D \bfu \in L^2(\Omega),  u_i\nu_i =0 \text{ on } \partial \Omega, i=1,2,3 \right\},$$
endowed with the inner product
$$
 \quad (u,v)_H : = \int_\Omega D \bfu \cdot \overline{D \bfv},
$$
and the following bilinear form defined on $H$:
$$
\alpha_\bftheta(u,v) : = \sum_{i \in \I^\bftheta}\int_{\Omega} a^{\rm hom}_i \partial_{x_i} u_i \cdot \overline{\partial_{x_i} v_i} = \int_\Omega A^{\hom} D \bfu \cdot \overline{D \bfv}, \qquad \bfu,\bfv \in {\rm\it H}_\bftheta.
$$
Note that for $\bftheta$ such that $\I^\bftheta = \{ i \in \I \, | \, \theta_i = 0\} = \emptyset$ then $H_\bftheta$ is zero and for such $\bftheta$ we define our `spatial' operator $A_\bftheta$ to be the zero map. Otherwise, $\alpha_\bftheta$ is a positive form on $H_\bftheta$ and therefore has a positive self-adjoint operator $A_\bftheta$, densely defined in $L^2(\Omega; \CC^\bftheta)$, associated with the form. The space $H$ is compactly embedded\footnote{This follows from an application of Vitali's theorem, which is permissible by noting that since $u_i$ has an $L^2$ weak derivative in the $x_i$-th direction one can use the fundamental theorem of calculus to prove that any bounded sequence in $H$ is 2-equi-integrable.} into $L^2$, and consequently the spatial operator $A_\bftheta$ has compact resolvent and therefore its spectrum is discrete. 
\subsection{Pure Bloch operators}
\label{sec:blochspec}

Consider the space
\begin{equation}
\label{contpartv}
\V_{{\bftheta}} = \{ v \in H^1_{{\bftheta}}(Q) \, | \,  \text{ $v \equiv 0$ on $Q_1$} \},
\end{equation}
which is a closed subspace of $H^1_{{\bftheta}}(Q)$, and therefore is a Hilbert space when  equipped with standard $H^1_{{\bftheta}}(Q)$ norm. Define the sesquilinear form 
$$
\beta_{\bftheta}(u,v): =   \int_{Q_0} a_0(y) \bfnabla_y u(y) \cdot \overline{\bfnabla_y v}(y) \, {\rm d}y , \quad u, v \in \V_{\bftheta}.
$$
Since $a_0$ is positive and bounded on $Q_0$, and elements of $\V_{\bftheta}$ have zero trace on the part of the boundary $\Gamma = \partial Q_1$, then by Poincar\'{e}'s inequality the form $\beta_{\bftheta}$ is (uniformly in ${\bftheta}$) coercive and bounded on $\V_{\bftheta}$, i.e. there exists $c_1$ and $c_2$ independent of ${\bftheta}$ such that
\begin{equation*}
\begin{aligned}
| \beta_{\bftheta}(u,v) | \le c_1 || u ||_{H^1_{\bftheta}}|| v ||_{H^1_{\bftheta}}, \\
\beta_{\bftheta}(u,u) \ge c_2 || u ||_{H^1_{\bftheta}}^2,
\end{aligned}
\end{equation*}
for all $u, v \in \V_{\bftheta}$. This implies that for every $f \in L^2(Q_0)$ there exists a unique solution $u \in \V_{\bftheta}$ such that
$$
\beta_{\bftheta}(u,v)= \int_{Q_0} f (y)  \overline{v}(y) \, {\rm d}y, \qquad \forall v \in \V_{\bftheta}.
$$
Consequently, the unbounded self-adjoint linear operator $B_{\bftheta}$, defined in $L^2(Q_0)$, given by $ B_{\bftheta} u =f$,  is positive and, moreover, by the Rellich embedding theorem has compact resolvent. Therefore the spectrum of $B_{\bftheta}$ is discrete, and we order the eigenvalues in accordance with the min-max principle. These eigenvalues can be shown to be continuous functions of ${\bftheta}$, in fact the following result holds.
\begin{lemma}
	\label{lem:conteigs}
	For each $n \in \NN$, let  $\mu^{(n)}_{\bftheta}$ denote the $n$-th eigenvalue of $B_{\bftheta}$ as ordered according to the min-max principle, i.e.
	\begin{equation}
	\label{minmax}
	\mu^{(n)}_{\bftheta} = \sup_{v_1, \ldots, v_{n-1} \in \V_{\bftheta}} \ \inf_{\substack{v \in \V_{\bftheta},\\ ||v||_{L^2(Q_0)}=1, \\ v \perp v_i, \forall i=1,\ldots,n}} \, \int_{Q_0} a_0 \bfnabla v \cdot \overline{\bfnabla v}, \qquad {\bftheta} \in [0,2\pi)^3,
	\end{equation}
	where $v \perp v_i$ is shorthand for $v$ is orthogonal to $v_i$ in $L^2(Q_0)$. Then, for each $n \in \NN$ the function $\lambda_n({\bftheta}) : = \mu^{(n)}_{\bftheta}$  is Lipschitz continuous, that is there exists a $C_n >0$ such that
	$$
	| \lambda_n({\bftheta}') - \lambda_n({\bftheta}) | \le C_n | {\bftheta}' - {\bftheta}|, \qquad \forall {\bftheta},{\bftheta}' \in [0,2\pi)^3.
	$$
\end{lemma}
The proof relies on an important observation is that the spaces $\V_{\bftheta}$, ${\bftheta} \in [0,2\pi)^3$, are mutually isomorphic. Indeed, if ${\bftheta},{\bftheta}' \in [0,2\pi)^3$ then it is clear that the isometric mapping $\mathcal{U}({\bftheta},{\bftheta}'): L^2(Q) \rightarrow  L^2(Q)$ defined as multiplication by the function $\exp\big( \i ({\bftheta}'-{\bftheta}) \cdot y\big)$ defines an isomorphism between $\V_{{\bftheta}}$ and $\V_{{\bftheta}'}$.
\begin{proof}
	Let $v$ be $L^2(Q_0)$-normalised element of $\V_{{\bftheta}}$ and consider $v' : = \mathcal{U}({\bftheta},{\bftheta}') v =  \exp\big( \i ({\bftheta}'-{\bftheta}) \cdot y \big) v$. Then, $v'$ is an $L^2(Q_0)$-normalised element of $\V_{{\bftheta}'}$ and the following identity
	\begin{flalign*}
	\int_{Q_0} a_0 \bfnabla v' \cdot \overline{\bfnabla v'}  = \int_{Q_0} a_0(y) & \bfnabla v(y) \cdot \overline{ \bfnabla v}(y) \, {\rm d}y + \int_{Q_0} a_0(y)    \bfnabla v(y) \cdot \overline{\i ({\bftheta}' - {\bftheta}) v}(y) \, {\rm d}y  \\
	&+  \int_{Q_0} a_0(y) \i ({\bftheta}' - {\bftheta})  \exp\big( \i ({\bftheta}'-{\bftheta}) y \big) v(y) \cdot \overline{\bfnabla v'}(y) \, {\rm d}y
	\end{flalign*}
	holds. Therefore, one has
	$$
	\begin{aligned}
	&\left\vert \int_{Q_0} a_0 \bfnabla v' \cdot \overline{\bfnabla v'} -  \int_{Q_0} a_0 \bfnabla v \cdot \overline{ \bfnabla v}  \right\vert \\ 
	&\hspace{1cm}\le || a_0 ||^{1/2}_{L^\infty(Q_0)} | {\bftheta}'- {\bftheta}| \left[ \left( \int_{Q_0} a_0 \bfnabla v \cdot \bfnabla v \right)^{1/2} + \left( \int_{Q_0} a_0 \bfnabla v' \cdot \bfnabla v' \right)^{1/2}  \right].
	\end{aligned}
	$$
	Consequently, as the isometric mapping $\mathcal{U}({\bftheta},{\bftheta}'): L^2(Q_0) \rightarrow L^2(Q_0)$ is an isomorphism between $\V_{\bftheta}$ and $\V_{{\bftheta}'}$, the above inequality and the min-max formula \eqref{minmax} implies that
	\begin{equation}\label{lipbound}
	| \lambda_n({\bftheta}') - \lambda_n({\bftheta}) | \le || a_0 ||^{1/2}_{L^\infty(Q_0)}| {\bftheta}'- {\bftheta}| \big( \lambda_n({\bftheta}') + \lambda_n({\bftheta}) \big).
	\end{equation}
	Now, if we consider the self-adjoint Dirichlet operator in $L^2(Q_0)$ associated with the form
	$$
	\beta_D(u,v) : = \int_{Q_0} a_0 \bfnabla u \cdot \bfnabla v, \qquad \forall u,v \in H^1_0(Q_0),
	$$
	then, since $H^1_0(Q_0)$ is embedded in $\V_{\bftheta}$ for all ${\bftheta}$, one has
	$$
	\lambda_n({\bftheta}) \le \mu_n : =\sup_{v_1, \ldots, v_{n-1} \in H^1_0(Q_0)} \ \inf_{\substack{v \in H^1_0(Q_0),\\ ||v||_{L^2(Q_0)}=1, \\ v \perp v_i, \forall i=1,\ldots,n}} \, \int_{Q_0} a_0 \bfnabla v \cdot \bfnabla v , \quad \forall {\bftheta}\in [0,2\pi)^3. 
	$$
	Here $\mu_n$ is the $n$-th eigenvalue\footnote{The spectrum of $B_D$ is discrete, which again is a consequence of the Rellich theorem.} of the operator $B_D$, defined in a similar manner as $B_{\bftheta}$ above. Hence, we deduce from \eqref{lipbound} that $\lambda_n({\bftheta})$ is Lipschitz continuous  with a Lipschitz constant bounded from above by $2 || a_0 ||^{1/2}_{L^\infty(Q_0)} \mu_n.$
\end{proof}
\section{Quasi-periodic two-scale limit spectrum}
\label{sec:limspec}
In this section we study the spectrum
$$
\bigcup_{\bftheta \in [0,2\pi)^3} \sigma(A^{\rm hom}_\bftheta).
$$
In particular we shall characterise the spectrum in terms of the spatial and pure Bloch operators introduced in Section \ref{sec:qplimop}. This leads to an appropriate analogue of the Zhikov $\beta$ function, cf. \cite{Zh1}.

Let us fix $\bftheta \in [0,2\pi)^3$ and suppose that $\lambda_{\bftheta}$ is in the spectrum of $A_{\bftheta}^{\rm hom }$. Then, there exists an eigenfunction $u_{\bftheta} \in V_{\bftheta}$ that solves the spectral problem
\begin{equation}\label{twoscaleqpspectral}\hspace{-2pt} \left\{ 
\begin{array}{lcr}
- {\rm div}_y\big( a_0 ( y ) \bfnabla_y u_{\bftheta}(x,y) \big) = \lambda_{\bftheta}  u_{\bftheta}(x,y), & & \  \text{$x \in \Omega$, $y \in Q_0$,}   \\[2pt]
\hspace{10pt} u_{\bftheta}(x,y) = u_i(x), \hspace{2cm}   & & x \in \Omega ,\  y \in \overline{C_i,} \\[2pt]
\text{where $u_i \equiv 0$ if ${\theta}_i \neq 0$ or otherwise solves} & &\\[2pt]
-  a^{\rm hom}_{i} \partial^2_{x_i} u_i(x) + \T_i(u_{\bftheta})(x) = \lambda | C_i | u_i(x), &  & \ \text{for $x \in \Omega$.}
\end{array} \right.
\end{equation} 
Here, we recall that
$$
\begin{aligned}
\T_i(u_{\bftheta})(x) = \int_{\Gamma_i} a_0(y) \bfnabla_y u_{\bftheta}(x,y) \cdot \bfn(y) \, {\rm d}S(y).
\end{aligned}
$$

There are two subcases to study: when ${\bftheta} \in \cup_{i \in \I} \Pi_i$, for $\Pi_i : = \{ {\bftheta} \in [0,2\pi^2)^3 \, | \, \bftheta \cdot \bfe_i = 0 \}$, and ${\bftheta} \in [0,2\pi)^3 \backslash \big( \cup_{i \in \I} \Pi_i \big)$.

\hspace{10pt} {\it Pure Bloch spectrum.} If  ${\bftheta} \in [0,2\pi)^3 \backslash \big( \cup_{i \in \I} \Pi_i \big)$, then $\lambda_\bftheta, u_{\bftheta}$ solves the problem 
\begin{equation}\label{twoscaleqpspectrale2}\hspace{-2pt} \left\{ 
\begin{array}{lcr}
- {\rm div}_y\big( a_0 ( y ) \bfnabla_y u_{\bftheta}(x,y) \big) = \lambda_{\bftheta}  u_{\bftheta}(x,y), & & \  \text{$x \in \Omega$, $y \in Q_0$,}   \\[2pt]
\hspace{10pt} u_{\bftheta}(x,y) = 0, \hspace{2cm}   & & x \in \Omega ,\  y \in \Gamma. 
\end{array} \right.
\end{equation} 
Therefore, setting $u_{\bftheta}(x,y) = \phi(x) v_{\bftheta}(y)$ for a sufficiently arbitrary $\phi$, we find that $v_{\bftheta}$ solves
\begin{equation}\label{twoscaleqpspectrale3}\hspace{-2pt} \left\{ 
\begin{array}{lcr}
- {\rm div}_y\big( a_0 ( y ) \bfnabla_y v_{\bftheta}(y) \big) = \lambda_{\bftheta}  v_{\bftheta}(y), & & \  \text{$y \in Q_0$,}   \\[2pt]
\hspace{10pt} v_{\bftheta}(y) = 0, \hspace{2cm}   & &  y \in \Gamma. 
\end{array} \right.
\end{equation} 
Therefore, the spectrum of $A_{\bftheta}^{\rm hom }$ for ${\bftheta} \in [0,2\pi)^3 \backslash \big( \cup_{i\in \I} \Pi_i \big)$ consists of eigenvalues of infinite multiplicity, and these eigenvalues coincide with the  eigenvalues the pure Bloch operator $B_{\bftheta}$ introduced in  Section \ref{sec:blochspec}.
Lemma \ref{lem:conteigs} implies that these eigenvalues are continuous with respect to ${\bftheta}$, and by continuously extending ${\bftheta}$ from $[0,2\pi)^3 \backslash \big( \cup_{i \in \I} \Pi_i \big)$ to $[0,2\pi)^3$  we deduce that
$$
\sigma( A^{\rm hom}) \supset \bigcup_{{\bftheta} \in [0,2\pi)^3} \sigma(B_{\bftheta}).
$$
It is for this reason that we call $ \bigcup_{{\bftheta} \in [0,2\pi)^3} \sigma(B_{\bftheta})$ the pure Bloch spectrum of $A^{\rm hom}_{\bftheta}$.

\hspace{10pt}{\it Spatial spectrum.} Let us now suppose that ${\bftheta} \in \cup_{i\in \I} \Pi_i $ and $\lambda_\bftheta \in \sigma(A_\bftheta)$ is not a pure Bloch eigenvalue, i.e. $\lambda_\bftheta \notin  \bigcup_{{\bftheta} \in [0,2\pi)^3} \sigma(B_{\bftheta})$. Introducing, for $i  \in \I^\bftheta$ the functions  $b^{(i)}_{\bftheta} \in H^1_{\bftheta}(Q)$ that satisfy
\begin{equation}
\label{microosc:b}
\left\{ \begin{aligned}
- {\rm div}_y\big( a_0 ( y ) \bfnabla_y b^{(i)}_{\bftheta}(y) \big) = 0, & & \  \text{$y \in Q_0$,}   \\[2pt]
\hspace{10pt} b^{(i)}_{\bftheta}(y) = \delta_{ij}, \hspace{2cm}   & &  y \in \overline{C_j}, \ j=1,2,3,
\end{aligned} \right.
\end{equation} 
we represent $u_{\bftheta}$ as follows
$$
u_{\bftheta}(x,y) = \sum_{i \in \I^{\bftheta}} u_i(x) b^{(i)}_{\bftheta}(y) + v_{\bftheta}(x,y),
$$
and substitute this representation into \eqref{twoscaleqpspectral} to deduce that $v_{\bftheta}(x,y) \in \V_{\bftheta}$, see \eqref{contpartv}, solves
\begin{equation}
\label{microosc:v}
\begin{aligned}
- {\rm div}_y\big( a_0 ( y ) \bfnabla_y v_{\bftheta}(x,y) \big) - \lambda_\bftheta v_{\bftheta}(x,y) = \lambda_\bftheta \sum_{i \in \I^{\bftheta}}  u_i(x) b^{(i)}_{\bftheta}(y), & & \  x \in \Omega, \text{$y \in Q_0$.}
\end{aligned} 
\end{equation} 
Denoting respectively by $\mu^{(m)}_{\bftheta}$ and $v^{(m)}_{\bftheta}$ the $m$-th eigenvalue and orthonormal eigenfunction of $B_{\bftheta}$, we perform a spectral decomposition of $v_\bftheta$ and $b^{(i)}_\bftheta$ to conclude that
$$
\begin{aligned}
v_\bftheta(x,y) = \sum_{m \in \NN} c_m(\bftheta,x)  v^{(m)}_{\bftheta}(y), &  & b^{(i)}_\bftheta(y) = \sum_{m \in \NN} b^{(i)}_m(\bftheta)v^{(m)}_{\bftheta}(y),
\end{aligned}
$$
for some $c_m(\bftheta,x)$ and constants
$
b^{(i)}_m(\bftheta) = \int_{Q_0} b^{(i)}_\bftheta \overline{v_\bftheta^{(m)}}.
$ Substituting the spectral representations into \eqref{microosc:v} gives
$$
\begin{aligned}
c_m(\bftheta,x) = \tfrac{\lambda_\bftheta}{\mu^{(m)}_{\bftheta} - \lambda_\bftheta} \sum_{i \in \I^{\bftheta}}u_i(x) b^{(i)}_m(\bftheta).
\end{aligned}
$$
Therefore, $u_{\bftheta}$ admits the form
$$
u_{\bftheta}(x,y) =\sum_{i \in \I^{\bftheta}} \sum_{m \in \NN} \Big( \tfrac{\mu^{(m)}_{\bftheta}}{\mu^{(m)}_{\bftheta} -\lambda_\bftheta} \Big) u_i(x) b^{(i)}_m(\bftheta) v^{(m)}_{\bftheta}(y), \qquad x \in \Omega, y \in Q_0.
$$
Consequently, we calculate
$$
\T_j(u_{\bftheta})(x) =\sum_{i \in \I^{\bftheta}} \sum_{m \in \NN} \Big( \tfrac{\mu^{(m)}_{\bftheta}}{\mu^{(m)}_{\bftheta} -\lambda_\bftheta} \Big) u_i(x) b^{(i)}_m(\bftheta) \int_{\Gamma_j} a_0\bfnabla_y v^{(m)}_{\bftheta}(y) \cdot \bfn(y) \, {\rm d}S(y).
$$
Recalling that $b^{(i)}_\bftheta$ solves \eqref{microosc:b}, $v_\bftheta^{(m)}$ solves \eqref{twoscaleqpspectrale3}, and utilising Green's identity, we deduce that
$$
\begin{aligned}
 \int_{\Gamma_j} a_0\bfnabla_y v^{(m)}_{\bftheta}(y) \cdot \bfn(y) \, {\rm d}S(y) & =  \int_{\Gamma} a_0\bfnabla_y v^{(m)}_{\bftheta}(y) \cdot \bfn(y) \overline{b^{(j)}_\bftheta}(y) \, {\rm d}S(y) 
=-\mu^{(m)}_{\bftheta} \int_{Q_0 }v^{(m)}_{\bftheta} \overline{b^{(j)}_\bftheta} \\
&=-\mu^{(m)}_{\bftheta} \overline{b^{(j)}_m}(\bftheta). 
\end{aligned}
$$
Therefore 
$$
\T_j(u_{\bftheta})(x) = -\sum_{i \in \I^{\bftheta}} \sum_{m \in \NN} \Big( \tfrac{ |\mu^{(m)}_{\bftheta}|^2}{\mu^{(m)}_{\bftheta} -\lambda_\bftheta} \Big) u_i(x) b^{(i)}_m(\bftheta) \overline{b^{(j)}_m(\bftheta)},
$$
and for each $i \in \I^{\bftheta}$, $u_i$ solves the problem
$$
- a^{\rm hom}_i \partial^2_{x_i} u_i(x) = \sum_{ \substack{ j \in \I^{\bftheta}  }} \beta^{(ij)}_{\bftheta}(\lambda_\bftheta) u_j(x) \qquad x \in \Omega, \qquad u_i \nu_i = 0 \text{ on $\partial \Omega$},
$$
for  
\begin{equation}
\label{matrixbeta}
\beta^{(ij)}_{\bftheta}(\lambda)  =\lambda | C_i | \delta_{ij} +  \sum_{m \in \NN} \Big( \tfrac{ |\mu^{(m)}_{\bftheta}|^2}{\mu^{(m)}_{\bftheta} -\lambda} \Big) b^{(j)}_m(\bftheta) \overline{b^{(i)}_m(\bftheta)}, \qquad \lambda \in \RR.
\end{equation}
Hence, we have demonstrated the following.
\begin{proposition}
\label{prop:limitspecchar}
	The spectrum of $\bigcup\limits_{\bftheta \in [0,2\pi)^3} \sigma(A_\bftheta^{\rm hom})$ is the union of the following two sets:
	\begin{itemize}
		\item{The pure Bloch spectrum: 
			$$
			\bigcup_{\bftheta \in [0,2\pi)^3} \sigma(B_\bftheta) = \sum_{m \in \NN} \left[ \min_{\bftheta \in [0,2\pi)^3} \big(\mu^{(m)}_\bftheta\big), \max_{\bftheta \in [0,2\pi)^3} \big( \mu^{(m)}_\bftheta \big) \right],
			$$
			where $\mu^{(m)}_\bftheta$ are the eigenvalues of $B_\bftheta$ ordered according to the min-max principle.}
		\item{The spatial spectrum: $\{ \lambda \in \ [0,\infty) \, | \,$ $\gamma(\bftheta) (\lambda) \in \sigma(A_\bftheta)\}$, where $A_\bftheta$ is an operator with compact resolvent.
			Here $\gamma : \RR^3 \rightarrow \mathbb{S}^3$ is for each $\bftheta$ a (possibly) sign-indefinite symmetric matrix defined by setting for $i \notin \I^\bftheta$, $\gamma_{ij}(\bftheta) =0$ for all $j$, and $\gamma_{ij}(\bftheta) = \beta^{(ij)}_\bftheta $ otherwise.
		}
	\end{itemize}
\end{proposition}

\section*{Acknowledgements}
This work was performed under the financial support of the Engineering and Physical Sciences Research Council Grant EP/M017281/1: ``Operator asymptotics, a new approach to length-scale interactions in metamaterials."


\begin{thebibliography}{9}
\bibitem{All}G. Allaire, {\it Homogenization and two-scale convergence}, SIAM J. Math. Anal. {\bf 23} (1992),
1482-1518.

\bibitem{AlCo}G. Allaire, C. Conca, {\it Bloch wave homogenization and spectral asymptotic analysis}, Journal de Math\'{e}matiques Pures et Appliqu\'{e}es, {\bf{ 77}}, Issue 2, (1998), 153-208.

\bibitem{Av1} A., Avila, G., Griso, B., Miara, {\it Bandes photoniques interdites en
\'{e}lasticit\'{e} lin\'{e}aris\'{e}e.} C. R. Math. Acad. Sci. Paris 339, (2004), 377–382.

\bibitem{Av2} A., Avila, G., Griso, B., Miara, E., Rohan, {\it Multiscale modeling of
elastic waves: theoretical justification and numerical simulation of
band gaps.} Multiscale Model. Simul. 7 {\bf 1}, (2008), 1–21.

\bibitem{BeGr} M. Bellieud, I., Gruais, {\it Homogenization of an elastic material
reinforced by very stiff or heavy fibers. Non-local effects. Memory
effects.} J. Math. Pure Appl. {\bf 84}, (2005), 55–96.

\bibitem{BoFe} G.,  Bouchitt\'{e}, D., Felbacq, {\it Homogenization near resonances and
artificial magnetism from dielectrics.} C. R. Math. Acad. Sci. Paris 339
{\bf 5}, (2004), 377–382.

\bibitem{Br} M., Briane, {\it Homogenization of non-uniformly bounded operators:
critical barrier for nonlocal effects.} Arch. Ration. Mech. Anal. 164 {\bf 1} (2002),
73–101.

\bibitem{CaMi} M., Camar-Eddine, G.W., Milton, {\it Non-local interactions in the
homogenization closure of thermoelectric functionals.} Asymptotic
Anal. 41 {\bf 3–4}, (2005), 259–276.

\bibitem{CaSe} M., Camar-Eddine, P., Seppecher, {\it Determination of the closure of the
set of elasticity functionals.} Arch. Ration. Mech. Anal. 170 {\bf 3}, (2003), 211–245.

\bibitem{cherd} M. I., Cherdantsev, {\it Spectral convergence for high contrast elliptic
periodic problems with a defect via homogenization.} Mathematika, Volume 55, Issue 1-2, (2009), pp. 29-57

\bibitem{ChCo1dhc} K. D.,  Cherednichenko, S.,  Cooper, S., Guenneau, {\it Spectral analysis of one-dimensional high-contrast elliptic problems with periodic coefficients.} Multiscale Modeling and Simulation: A SIAM journal, Volume 13, Issue 1, (2015), pp.72-98.

\bibitem{ChCo2} K. D., Cherednichenko, S., Cooper {\it  Homogenisation of the system of high-contrast Maxwell equations.}  Mathematika, Volume 61, Issue 2, (2015), pp.475-500.

\bibitem{ChCo} K. D., Cherednichenko, S., Cooper {\it  Resolvent estimates for high-contrast elliptic problems with periodic coefficients.} Archive for Rational Mechanics and Analysis, Volume 219, Issue 3, (2016), pp.1061-1086.

\bibitem{ChSmZh} K. D., Cherednichenko, V. P., Smyshlyaev, V. V., Zhikov, {\it  Non-local
homogenised limits for composite media with highly anisotropic
periodic fibres.} Proc. R. Soc. Edinb. A 136 {\bf 1}, (2006), 87–114.


\bibitem{Cothesis} S., Cooper, {\it Two-scale homogenisation of partially degenerating PDEs with applications to photonic crystals
and elasticity.} PhD Thesis, (2012), University of Bath.

\bibitem{Co13} S., Cooper, {\it Homogenisation and spectral convergence of a periodic elastic composite with weakly compressible inclusions.} Applicable Analysis: An International Journal, 93 (7), (2013).

\bibitem{FeKh} V. N., Fenchenko, E. Ya. Khruslov, {\it Asymptotic behaviour or the
solutions of differential equations with strongly oscillating and
degenerating coefficient matrix.} Dokl. Akad. Nauk Ukrain. SSR Ser. A
{\bf 4} (1980), 26–30.

\bibitem{HeLi} R., Hempel, K., Lienau, {\it  Spectral properties of periodic media in the
large coupling limit.} Commun. Part. Diff. Eq. 25 {\bf 7–8}, (2000), 1445–1470.

\bibitem{JKO} V. V. Jikov, S. M. Kozlov, O.A. Oleinik, {\it Homogenization of Differential Operators and Integral Functionals.}

\bibitem{KaSm1} I. V., Kamotski, V. P., Smyshlyaev, {\it Localised modes due to defects in
high contrast periodic media via homogenization.} Available on ResearchGate: https://www.researchgate.net/publication/253995046\_Localised\_modes\_due
\_to\_defects\_in\_high\_contrast\_periodic\_media\_via\_homogenization.

\bibitem{IVKVPS} I. V. Kamotski, V. P. Smyshlyaev, {\it Two-scale homogenization for a class of partially degenerating PDE systems.} {\it Preprint arXiv:1309.4579v1,} (2013).

\bibitem{KoSh} R.V., Kohn, S.P., Shipman, {\it Magnetism and homogenization of
microresonators.} Multiscale Model. Simul. 7 {\bf 1}, (2008), 62–92.

\bibitem{Ku} P. Kuchment, {\it  The mathematics of photonic crystals.} In Mathematical modeling in optical science, volume 22
of Frontiers Appl. Math., pages 207-272. SIAM, Philadelphia, PA, 2001.

\bibitem{LiChSh} Z., Liu, C.T., Chan, P., Sheng, {\it Analytic model of phononic crystals
with local resonances.} Phys. Rev. B 71, (2005), 014103.


\bibitem{MiBrWi} G, Milton, M., Briane, J., Willis, {\it On cloaking for elasticity and
physical equations with a transformation invariant form.} New J. Phys. 8, (2006), 248.

\bibitem{MiWi} G., Milton, J., Willis, {\it  On modifications of Newton’s second law
and linear continuum elastodynamics.} Proc. R. Soc. Lond. A {\bf 463}, (2007), 855–
880.

\bibitem{Ng}G. Nguetseng, {\it A general convergence result for a functional related to the theory of
homogenization}, SIAM J. Math. Anal. {\bf 20} (1989), 608-623.

\bibitem{Sa} G. V., Sandrakov, {\it Homogenization of elasticity equations with
contrasting coefficients.} Sbornik Math. 190 {\bf 12} (1999), 1749–1806.

\bibitem{SmMoM} V. P., Smyshlyaev, {\it Propagation and localization of elastic waves in highly anisotropic
	periodic composites via two-scale homogenization}. Mechanics of Materials, 41, (2009), pp 434-447. 

\bibitem{Zh1}V. V., Zhikov {\it On an extension of the method of two-scale convergence and its applications}, Sbornik: Mathematics {\bf 191}:7 (2000), 973-1014.

\bibitem{Zh2}V. V., Zhikov {\it On spectrum gaps of some divergent elliptic operators  with periodic coefficients}, St. Petersburg Math. J. {\bf 16}:5 (2005), 773-790.

\bibitem{ZhPa} V. V., Zhikov, S. E. Pastukhova {\it On gaps in the spectrum of the operator of elasticity theory on a high contrast periodic structure}. Journal of Mathematical Sciences, Volume 188, Issue 3, (2013), pp 227-240.

\end{thebibliography}
\end{document}